\documentclass[8pt]{article}
\usepackage{indentfirst,latexsym,bm}
\usepackage{amsfonts}
\usepackage{amssymb}
\usepackage[leqno]{amsmath}
\usepackage{dsfont}
\usepackage{amsthm}
\usepackage[all]{xy}
\usepackage{hyperref}
\usepackage{color,xcolor}
\begin{document}
\title{On slightly degenerate fusion categories}
\author{Zhiqiang Yu}
\date{}
\maketitle

\newtheorem{theo}{Theorem}[section]
\newtheorem{prop}[theo]{Proposition}
\newtheorem{lemm}[theo]{Lemma}
\newtheorem{coro}[theo]{Corollary}
\theoremstyle{definition}
\newtheorem{defi}[theo]{Definition}
\newtheorem{example}[theo]{Example}
\newtheorem{remk}[theo]{Remark}
\newtheorem{conj}[theo]{Conjecture}

\newcommand{\A}{\mathcal{A}}
\newcommand{\B}{\mathcal{B}}
\newcommand{\C}{\mathcal{C}}
\newcommand{\D}{\mathcal{D}}
\newcommand{\E}{\mathcal{E}}
\newcommand{\I}{\mathcal{I}}
\newcommand{\M}{\mathcal{M}}
\newcommand{\N}{\mathcal{N}}
\newcommand{\Q}{\mathcal{O}}
\newcommand{\Z}{\mathbb{Z}}
\theoremstyle{plain}

\abstract
In this paper, we first show for a slightly degenerate pre-modular fusion category $\C$
that  squares of dimensions of simple objects divide half of the dimension of $\C$,
and that slightly degenerate fusion categories of FP-dimensions $2p^nd$
and $4p^nd$ are nilpotent, where $p$ is an odd prime and $d$ is an odd square-free integer. Then we classify slightly degenerate generalized Tambara-Yamagami fusion categories and  weakly integral slightly degenerate fusion categories of particular dimensions.


\section{Introduction}
Throughout this paper, we always assume $k$ is an algebraically closed field of characteristic zero, $k^{*}:=k\backslash {\{0}\}$, $\mathbb{Z}_r:=\mathbb{Z}/r$, $r\in \mathbb{N}$. Categories are assumed to be semisimple $k$-linear finite abelian categories. For any finite abelian category $\C$, let $\Q(\C)$  denote set of isomorphism classes of simple objects of $\C$, the cardinal of $\Q(\C)$ is called rank of $\C$.

Pre-modular fusion category $\C$ is slightly degenerate (or super modular) if the  M\"{u}ger center $\C'$ is equivalent to $sVec$ as symmetric fusion category, where $sVec$ is the category of finite-dimensional super vector spaces. Slightly degenerate fusion categories   play a crucial role in classification of braided fusion categories, for example,  let $\E=Rep(G)\subseteq\C'$  be a maximal Tannakian subcategory, then the de-equivariantization $\C_G$ of braided fusion category $\C$ by $\E$ is either non-degenerate or slightly degenerate \cite{DrGNO2},  moreover it  is also useful in physics, see \cite{BGHNPRW,BGNPRW1,BGNPRW} for references.

The problem of classifying slightly degenerate fusion categories has been considered in several papers. For pointed slightly degenerate fusion categories, they are Deligne tensor products of $sVec$ and pointed non-degenerate categories \cite{ENO2}. If the  rank of a slightly degenerate fusion category is equal to $2$, then obviously it is equivalent to $sVec$. Grothendieck equivalence classes of slightly degenerate fusion category of rank $4$ and $6$ have been classified in \cite{B,BGNPRW}, respectively. In  \cite{DNO}, the authors studied the Witt group of slightly degenerate fusion categories, and the structure of  completely anisotropic slightly degenerate fusion categories  are obtained.

In this paper, we  classify slightly degenerate fusion categories of FP-dimensions like $2p^nd$, $4p^nd$, $2^m$ $(m\leq5)$, $8d$ and $16d$, where $p$ is an odd prime and $d$ is an odd square-free integer, and also we study slightly degenerate generalized Tambara-Yamagami fusion categories \cite{Lip,TY}. Most of these categories are  Deligne tensor products of $sVec$ and non-degenerate fusion categories of smaller FP-dimensions, the rest of them are Deligne tensor products of pointed non-degenerate fusion categories and $\Z_2$-equivariantizations of  generalized Tambara-Yamagami fusion categories. In particular, integral slightly degenerate fusion categories of FP-dimensions less than $64$ are pointed. Meanwhile we also obtain weakly group-theoretical property  of weakly integral braided fusion category $\C$, which satisfies integral simple objects have  FP-dimension $p^mq^n$ Proposition \ref{simplep^mq^nsolv}, where $p, q$ are primes, $m,n$ are non-negative integers.

The organization of this paper is  as follows. In section \ref{section2}, we recall some basic properties of fusion categories that we use throughout, and we prove  for any strictly weakly integral slightly degenerate fusion category $\C$ that $8\mid FPdim(\C)$ in Theorem \ref{slightstrictweakly}. In section \ref{section3}, we show for  super-modular fusion category $\C$ that  $\frac{dim(\C)}{2dim(X)^2}$ ($X\in\Q(\C)$) is an algebraic integer in Corollary \ref{sqdimension} and that the slightly degenerate fusion categories of FP-dimension $2p^nd$ and $4p^nd$ are nilpotent in Corollary \ref{slightly2p^nd4p^nd}, where $p$ is an odd prime and where $d$ is an odd square-free integer. In section \ref{section4}, we first classify slightly degenerate generalized Tambara-Yamagami fusion categories in Corollary \ref{generalizedTYstru}, and then we  classify   slightly degenerate fusion categories of FP-dimensions $2^n$ $(n\leq5)$ and $8d$ in Theorem \ref{slightly8d} and Theorem \ref{slightly16d}, where $d$ is a square-free integer.

\section{Preliminaries}\label{section2}
In this paper, we use $Vec$ and $sVec$ to denote the category of finite-dimensional vector spaces and super vector spaces  over $k$, respectively. We will freely use the basic theory of fusion categories and FP-dimensions but will recall some most used facts below, we refer the reader to \cite{DrGNO2,EGNO,ENO1,ENO2} for properties of fusion categories and braided fusion categories.

Let $G$ be a finite group, for a $G$-graded fusion category $\C=\oplus_{g\in G}\C_g$, we denote $FPdim(\C_g):=\sum\limits_{\substack{X\in\Q(\C_g)}}FPdim(X)^2$, $\forall g\in G$. In particular, if the grading is faithful, that is, $\C_g$ is non-zero, $\forall g\in G$, then $FPdim(\C)=|G|FPdim(\C_g)$ \cite[Proposition 8.20]{ENO1}. Let $\C_{pt}$ be the maximal pointed fusion subcategory of $\C$, $G(\C)$ be the group of isomorphism classes of invertible objects of $\C$. Let  $\C_{ad}$ be the  adjoint fusion subcategory of $\C$, i.e. $\C_{ad}$  is  generated by simple objects $Y$ such that $Y\subseteq X\otimes X^*$ for some simple object $X$, then $\C$ has a faithful grading with $\C_{ad}$ be the trivial component, the grading is called universal grading of $\C$, grading group is  denoted by $U_\C$ \cite [Corollary 3.7]{GN}.

A fusion category $\C$ is weakly integral if $FPdim(\C)\in\Z$; $\C$ is integral if all the simple objects are integral, that is $FPdim(X)\in\Z$, $\forall X\in\Q(\C)$; $\C$ is strictly weakly integral if there exists a non-integral simple object $X$, i.e. $FPdim(X)\notin\Z$.

For a weakly integral fusion category $\C$, the FP-dimension of simple objects are square-roots of integers, and if $\C$ is strictly weakly integral, then $\C_{ad}$ is integral by
\cite[Proposition 8.27]{ENO1}. The FP-dimension defines a faithful grading on $\C$, the grading group $E$ is an elementary abelian $2$-group \cite[Theorem 3.10]{GN}, this grading is called dimension grading. Let $\C_{int}$ be  the trivial component of the dimension grading of $\C$, which  is  the maximal integral fusion subcategory of $\C$, therefore for a weakly integral fusion category $\C$, we have $\C_{pt}\subseteq\C_{int}$ and $\C_{ad}\subseteq\C_{int}$.

Fusion category $\C$ is nilpotent if there exists  a natural number $n$ such that $\C^{(n)}=Vec$, where $\C^{(0)}:=\C$, $\C^{(1)}:=\C_{ad}$, $\C^{(j)}:=(\C^{(j-1)})_{ad}$, $j\geq1$ \cite{GN}. Pointed fusion categories and fusion categories of prime power FP-dimensions are nilpotent \cite [Theorem 8.28]{ENO1}, for example. Fusion category $\C$ is weakly group-theoretical if it is Morita equivalent to a nilpotent fusion category $\D$ \cite{ENO2}. Fusion category $\C$ is group-theoretical if it is Morita equivalent to a pointed  fusion category $\D$.

Fusion category $\C$ is a braided  fusion category if for any $X,Y,Z\in\C$, there exists a natural isomorphism $c_{X,Y}:X\otimes Y\to Y\otimes X$, and braiding $c$ satisfies $c_{X,I}=c_{I,X}=id_X$, $c_{X\otimes Y,Z}=c_{X,Z}\otimes id_Y \circ id_X\otimes c_{Y,Z}$, $c_{Z,X\otimes Y}=id_X\otimes c_{Z,Y}\circ c_{Z,X}\otimes id_Y$, here we suppress the associativity isomorphism. Braided fusion category $\C$ is a pre-modular  (or ribbon)  fusion category if $\C$ admits a spherical structure. It follows from \cite[Proposition 8.23, Proposition 8.24]{ENO1} that any  weakly integral braided fusion category $\C$ is pre-modular, and $dim(X)=FPdim(X)$, $\forall X\in\Q(\C)$, where $dim(X)$ is taking trace of morphism $id_X$ via spherical structure of $\C$. Throughout this paper, for any fusion category $\C$, we always assume $dim(X)=FPdim(X)$, $\forall X\in\Q(\C)$, unless otherwise stated.

For a pre-modular fusion category $\C$, let $\theta$ be the ribbon structure of $\C$.
Then for any simple object $X\in\C$, $\theta_X\in Hom_\C(X,X)\cong k$ will be regarded as a non-zero scalar. The $T$-matrix $T=(T_{X,Y})_{X,Y\in\Q(\C)}$ is the diagonal matrix with coefficients $T_{X,Y}:=\delta_{X,Y}\theta_X$;
the $S$-matrix $S=(s_{X,Y})_{X,Y\in\Q(\C)}$ is defined via the spherical structure of $\C$. Specifically,  $s_{X,Y}:=Tr(c_{Y,X} c_{X,Y})$. Pre-modular fusion category $\C$ is modular if $S$-matrix is non-degenerate.

For a braided fusion category $\C$, let $\D\subseteq\C$ be a fusion subcategory, the  centralizer  $\D'$ of $\D$ is the fusion subcategory generated by all simple objects $X$ such that $c_{Y,X} c_{X,Y}=id_{X\otimes Y}$, $\forall Y\in\D$. Braided fusion category $\C$ is non-degenerate if and only if $\C'=Vec$ as symmetric fusion category by \cite[Proposition 3.7]{DrGNO2}; moreover, modular fusion category $\C$ is  exactly non-degenerate with a spherical structure. $\C$ is slightly degenerate (or super modular), if M\"{u}ger center $\C'\cong sVec$ as symmetric fusion category. We use $\chi$ to denote the generator of $sVec$ throughout, which is called fermion in \cite{Mu}.

Fusion subcategory $\D\subseteq\C$ is symmetric if $\D\subseteq\D'$.
By \cite[Corollaire 0.8]{De} any symmetric fusion category is braided equivalent to $Rep(G,u)$, where $Rep(G,u)$ is the category of finite-dimensional super representations of finite group $G$, $u\in G$ is a central element of order $2$ and acts as parity automorphism. In general, we say that $Rep(G,u)$ is super-Tannakian. Symmetric fusion category $\C$ is Tannakian if $\C\cong Rep(G)$ as symmetric fusion category for a finite group $G$. So symmetric fusion categories of odd FP-dimensions are always Tannakian categories. Besides, symmetric fusion category $\C\cong Rep(G,u)$ contains a maximal Tannakian subcategory $\E\cong Rep(G/\langle u\rangle)$ with $FPdim(\E)=\frac{1}{2}FPdim(\C)$
\cite[Corollary 2.50]{DrGNO2}.

For braided fusion category $\C$, let $\E=Rep(G)\subseteq\C$ be a non-trivial Tannakian subcategory, then there exists a fusion category $\C_G$, which is called de-equivariantization of $\C$ by $\E$, and $\C\cong (\C_G)^G$ as  fusion categories. In general, $\C_G$ is  a braided $G$-crossed fusion category, which admits a $G$-grading, and the trivial component of the $G$-grading $\C_G^0\cong \E'_G$ is a braided fusion category. If particular, if $\E\subseteq \C'$, then $\C_G$ is a braided fusion category, and $(\C_G)'\cong (\C')_G$, see \cite[section 4]{DrGNO2} for details.

\begin{theo}[\cite{DrGNO2}, Theorem 3.14]\label{dimensionform}For any braided fusion category $\C$, let $\D\subseteq\C$ be a fusion subcategory, then $FPdim(\D)FPdim(\D')=FPdim(\C)FPdim(\D\cap\C')$.
\end{theo}
\begin{coro}\label{Tannakiandimen}For a slightly degenerate  fusion category $\C$, let $\E$ be a Tannakian subcategory of $\C$, then the ratio $\frac{FPdim(\C)}{FPdim(\E)^2}$ is an algebraic integer. In particular, if $\C$ is weakly integral, then $\frac{FPdim(\C)}{FPdim(\E)^2}$ is an integer.
\end{coro}
\begin{proof}Since $\C'\cong sVec$ and $\E$ is Tannakian,  $\E\cap\C'=Vec$. By definition $\E\subseteq\E'$, then by \cite[Proposition 8.15]{ENO1} $m=\frac{FPdim(\E')}{FPdim(\E)}$ is an algebraic integer. Therefore $m=\frac{FPdim(\C)}{FPdim(\E)^2}$ is an algebraic integer by Theorem \ref{dimensionform}.
\end{proof}
The following lemma is well-known and will be useful throughout this paper.
\begin{lemm}[\cite{Mu}, Lemma 5.4]\label{slightMugercenter}For a slightly degenerate fusion category $\C$, let $\chi$ be the generator of $sVec=\C'$, then $\chi\otimes X\ncong X$, $\forall X\in\Q(\C)$.
\end{lemm}
\begin{prop}[\cite{DrGNO2}, Proposition 3.29]\label{admugercenter}For a slightly degenerate fusion category $\C$, we have $(\C_{ad})'=\C_{pt}$, and $(\C_{pt})'=(\C_{ad})''=\C_{ad}\vee\C'$.
\end{prop}
\begin{remk}\label{Tannmugdeequiv} Indeed, for any braided fusion category $\C$, let $\D\subseteq\C$ be a fusion subcategory, then $\D''=\D\vee\C'$ \cite[Corollary $3.11$]{DrGNO2}, where  $\D\vee\C'$ is  the fusion subcategory generated by $\D$ and $\C'$. In particular, if $\C$ is slightly degenerate, and  $\E=Rep(G)\subseteq\C$ is  a Tannakian subcategory, as $\E''=\E\vee\C'$, by Theorem \ref{dimensionform} $FPdim(\E'')=2FPdim(\E)$. Therefore  \cite[Proposition 4.30 (iii)]{DrGNO2} implies $(\E'_G)'\cong(\E'\cap\E'')_G=(\E\vee\C')_G\cong sVec$, that is, $\E'_G$ is also slightly degenerate.
\end{remk}

\begin{prop}[\cite{DongNa}, Corollary 3.2]\label{striweakintchar}Let $\C$ be weakly integral braided fusion category, if $4\nmid FPdim(\C)$, then $\C$ is integral.
\end{prop}
For weakly integral slightly degenerate fusion categories, we can improve  Proposition \ref{striweakintchar} a little further. We begin with a special case.
\begin{lemm}\label{C'subseteqC_ad}Let $\C$ be a strictly  weakly integral slightly degenerate fusion category, if $\C'\subseteq \C_{ad}$, then $8|FPdim(\C)$.
\end{lemm}
\begin{proof}Assume $FPdim(\C)=2^md$, with $d$ is an odd integer, and $|E|=2^r$, $r\geq1$, as $E$ is an elementary abelian $2$-group \cite[Theorem 3.10]{GN}.  Under the assumption $\C'\subseteq \C_{ad}$,  it follows from Theorem \ref{dimensionform} and Proposition \ref{admugercenter} that \begin{align*}
FPdim(\C_{pt})=\frac{2FPdim(\C)}{FPdim(\C_{ad})}=2|U_\C|.
 \end{align*}
 Universal property of  $U_\C$ implies that there is a  surjective group homomorphism $U_\C\to E$ by \cite[Corollary 3.7]{GN}, then $|U_\C|=|E|a=2^ra$, $a\in\Z$. As $\C_{pt}\subseteq\C_{int}$, we have  $FPdim(\C_{int})=FPdim(\C_{pt})b=2^{r+1}ab$ by \cite[Proposition 8.15]{ENO1}, $b\in\Z$. Meanwhile faithfulness of the dimension grading  means $FPdim(\C)=|E|FPdim(\C_{int})=2^{2r+1}ab$. Then $2^{m-(2r+1)}d=ab\in\Z$, which implies $m\geq2r+1\geq3$.
\end{proof}

Recall that a symmetric fusion category $\C$ is super-Tannakian, if  there exists a finite group $G$ and a central element $u\in G$ of order $2$, such that $\C\cong Rep(G,u)$ as a symmetric category, where $u$ acts as parity automorphism.
\begin{theo}\label{slightstrictweakly}Let $\C$ be a strictly  weakly integral braided fusion category such that $\C'$ is super-Tannakian,  then $8| FPdim(\C)$.
\end{theo}
\begin{proof}It suffices to prove the theorem for $\C'=sVec$. Indeed, if $\C'$ is a super-Tannakian fusion category, then $\C'$ contains a maximal Tannakian subcategory $E=Rep(G)$
\cite[Corollary 2.50]{DrGNO2}, by Remark \ref{Tannmugdeequiv} $\C_G$ is a slightly degenerate strictly weakly integral fusion category, and $FPdim(\C)=FPdim(\C_G)|G|$. So if $8| FPdim(\C_G)$, then  $8|FPdim(\C)$.

Let $\C$ be a weakly integral  slightly degenerate fusion category with $FPdim(\C)=2^md$, by Proposition \ref{striweakintchar} we only need to show that if $m=2$  then $\C$ is integral. Note that if $d=1$, this is trivial, below we assume $d>1$. On the contrary, assume $\C$ is strictly weakly integral, then dimension grading group $E\cong\Z_2$ or $\Z_2\times\Z_2$ as $|E||FPdim(\C)$. While $sVec=\C'\subseteq \C_{int}$, therefore $E\cong\Z_2$ and $FPdim(\C_{int})=2d$. Since $\C'\subseteq\C_{pt}\subseteq\C_{int}$, we can assume $FPdim(\C_{pt})=2t$ with $t|d$. Let $p$ be an odd prime, if $p|FPdim(\C_{pt})$, then $\C_{pt}$ contains a braided fusion category $\A$ of FP-dimension $p$, from $\A\cap\C'=Vec$ we see that $\A$ is either a modular fusion subcategory or a Tannakian subcategory.

If $\C_{pt}$ does not contain any non-trivial Tannakian subcategory, then $\A$ is modular and \cite[Theorem 3.13]{DrGNO2} implies $\C\cong\A\boxtimes \A'$ with $\A'$ is slightly degenerate strictly weakly integral, so we can replace $\C$ by $\A'$. Therefore, we further assume $FPdim(\C_{pt})=2$, that is, $\C_{pt}=\C'=sVec$. Then $\C_{ad}\cap\C_{pt}=Vec$, otherwise $\C$ is integral by Lemma \ref{C'subseteqC_ad}. However that means $\C_{ad}$ is a non-trivial non-degenerate fusion category by Proposition \ref{admugercenter}, and $\C\cong\C_{ad}\boxtimes sVec$ by \cite[Theorem 3.13]{DrGNO2}, which implies $\C$ is integral, contradiction. Hence $\C_{pt}$ (then $\C$) must contains a non-trivial Tannakian subcategory.

Let $\E\cong Rep(G)$ be a maximal Tannakian subcategory of $\C$, then the core $\D:=\C_G^0\cong \E'_G$ of $\C$ is a slightly degenerate fusion category by Remark \ref{Tannmugdeequiv}, hence $2\nmid FPdim(\E)$, otherwise $2\nmid FPdim(\D)$ and then $sVec\nsubseteq\D$. Therefore $4\mid FPdim(\E')$ by Theorem \ref{dimensionform}. Note $\D$ is weakly anisotropic by \cite[Corollary 5.19]{DrGNO2}, then  \cite[Corollary 5.29]{DrGNO2} implies $\D_{pt}\cap\D_{ad}\subseteq\D_{pt}\cap(\D_{pt})'=\D'\cong sVec$ as $(\D_{pt})'=\D_{ad}\vee\D'$. If $\D'\subseteq\D_{ad}$, by Lemma \ref{C'subseteqC_ad}, $\D$ is integral; if $\D_{pt}\cap\D_{ad}=Vec$, then $\D_{ad}=Vec$ or $\D_{ad}$ is a non-trivial modular category, so $\D$ is pointed or $\D\cong\D_{ad}\boxtimes sVec$ by \cite[Theorem 3.13]{DrGNO2}. From both cases we obtain that $\D$ is an integral fusion category. However, integrality of $\D$ means that $\E'\cong\D^G$ is integral by \cite[Corollary 4.27]{DrGNO2}, and  then $4|FPdim(\C_{int})$ for $\E'\subseteq\C_{int}$, this is impossible. Hence $\C$ is integral.
\end{proof}
\begin{remk}The condition that  $\C'$ is super-Tannakian  can not be dropped, for example,  let $\C=\I$ be a braided Ising category, then $\C$ is strictly weakly integral and modular
\cite[Corollary B.12]{DrGNO2}, but $FPdim(\C)=4$. And there exists slightly degenerate  strictly weakly integral fusion category of FP-dimension $8$, that is $\I\boxtimes sVec$.
\end{remk}

\section{Nilpotency of a class of slightly degenerate fusion categories}\label{section3}
In this section, we first show for a super-modular fusion category $\C$ that $\frac{dim(\C)}{2dim(X)^2}$ is an algebraic integer for any $X\in\Q(\C)$, and then we prove nilpotency of slightly degenerate fusion categories of FP-dimensions $2p^nd$ and $4p^nd$, where $p$ is an odd prime, $d$ is an odd square-free integer, and $(p,d)=1$.

As we require that $dim(X)=FPdim(X)$ for $X\in\Q(\C)$, then for $sVec$, as a symmetric fusion category, $\theta_\chi=-1$, the  $S$-matrix of $sVec$ is $S=\left(
\begin{array}{cc}
  1 & 1 \\
 1 & 1 \\
\end{array}
\right)$.

By Lemma \ref{slightMugercenter}, for a slightly degenerate fusion category $\C$, rank of fusion subcategory $\A$ such that $\C'\subseteq\A$ is even. There exists a unnatural partition of $\Q(\C)$, which divides  $\Q(\C)$ into two subsets $\Pi_0\cup\Pi_1$ with same cardinal: unit object $I\in\Pi_0$; if $X\in\Pi_0$ then $X^*\in\Pi_0$, and there is a unique $Y\in\Pi_1$ s.t. $Y\cong \chi\otimes X$ \cite{BGHNPRW}. As $\C'=\langle\chi\rangle$, the balancing equation [e.g. \cite{EGNO}, Proposition 8.13.8] then implies $\theta_X=-\theta_Y$, if $Y\cong \chi\otimes X$.

For slightly degenerate pointed fusion categories, their structures are well-known.
\begin{prop}[\cite{DrGNO2,ENO2}]\label{slighpointed}Let $\C$ be slightly degenerate pointed fusion category, then $\C\cong\D\boxtimes sVec$, where $\D$ is a pointed non-degenerate fusion category.
\end{prop}

The following proposition was first proved in \cite[Corollary 2.7]{ENO2}, see also \cite[Theorem 3.9]{BGHNPRW}.

\begin{prop}For a slightly degenerate fusion category $\C$ with partition $\Q(\C)=\Pi_0\cup\Pi_1$, its $S$-matrix $S=\left(
                                                              \begin{array}{cc}
                                                                \widehat{S} & \widehat{S} \\
                                                                \widehat{S} & \widehat{S} \\
                                                              \end{array}
                                                            \right)$, where $\widehat{S}$ is a non-degenerate matrix with orthogonal rows and columns.
\end{prop}
In \cite{BGHNPRW,BGNPRW}, for a slightly degenerate fusion category $\C$ with partition $\Q(\C)=\Pi_0\cup\Pi_1$, there is a naive fusion rule $\widehat{N}={\{\widehat{N}_{X}|X\in\Pi_0}\}$, $\widehat{N}_{X}=((\widehat{N}_{X})_{Z,Y})$, where $(\widehat{N}_{X})_{Z,Y}=\widehat{N}^Z_{X,Y}:=N^Z_{X,Y}+N^{Z\otimes \chi}_{X,Y}$, $X, Y,Z\in \Pi_0$, $N^Z_{X,Y}:=dim_kHom_\C(X\otimes Y,Z)$.

Note that for any pre-modular category $\C$, $s_{X,Y^*}=\overline{s_{X,Y}}$, $X,Y\in\Q(\C)$. Just like modular case, we have  the following proposition:
\begin{prop}[\cite{BGNPRW}, Proposition 2.7]\label{orthomatrix}For a super-modular fusion category $\C$, the following  hold:
\begin{description}
  \item[(1)] $\widehat{S}\overline{\widehat{S}}=\frac{1}{2}dim(\C) E$, E is the unit matrix, i.e. $\widehat{S}$ is projectively unitary.
  \item[(2)]$\widehat{N}$ is a commutative fusion rule, that is $\widehat{N}_X\widehat{N}_Y=\widehat{N}_Y\widehat{N}_X$, $X,Y\in\Pi_0$.
  \item[(3)]The functions $\varphi_Y(X):=\frac{\widehat{s}_{X,Y}}{\widehat{s}_{Y,I}}$ $(Y\in \Pi_0)$ form a set of orthogonal characters of the algebra generated by $\Pi_0$ with naive fusion rule.
  \item[(4)]$\widehat{N}^Z_{X,Y}=\Sigma_{V\in \Pi_0}\frac{\frac{2}{dim(\C)}\widehat{s}_{X,V}\widehat{s}_{V,Y}\widehat{s}_{Z,V^*}}{\widehat{s}_{V,I}}$, $X,Y,Z\in\Pi_0$.
\end{description}
\end{prop}

Note that in Proposition \ref{orthomatrix}  (and corollary below), for simple object $X$ of a spherical fusion category $\C$, we do not require $dim(X)=FPdim(X)$. For any pre-modular fusion category $\C$, $\forall X\in\Q(\C)$, the ratio $\frac{dim(\C)}{dim(X)}$ is an algebraic integer; moreover if $\C$ is a  modular fusion category, then ratio $\frac{dim(\C)}{dim(X)^2}$ is an algebraic integer, see \cite[$\S 8.14$, $\S9.3$]{EGNO} for details, for example. Now we strengthen this property for super-modular fusion categories.

\begin{coro}\label{sqdimension}For a super-modular fusion category $\C$,  $\frac{dim(\C)}{2dim(X)^2}$ is an algebraic integer, $\forall X\in \mathcal{O}(\C)$.
\end{coro}
\begin{proof}The proof is same as modular case (see e.g.  \cite[Proposition 8.14.6]{EGNO}), we include it for the reader's convenience. It is sufficient to prove this for $X\in\Pi_0$. By Proposition \ref{orthomatrix} $(1)$, for all $ X\in\Pi_0$, we have
\begin{align*}
\frac{dim(\C)}{2dim(X)^2}=\frac{\Sigma_{Y\in \Pi_0}s_{X,Y}s_{Y,X^*}}{s^2_{I,X}}=\Sigma_{Y\in \Pi_0}\frac{s_{X,Y}}{s_{I,X}}\frac{s_{Y,X^*}}{s_{I,X}}=\Sigma_{Y\in \Pi_0}\frac{s_{X,Y}}{s_{I,X}}\frac{s_{Y,X^*}}{s_{I,X^*}},
 \end{align*}
 the last equality is because $s_{I,X}=dim(X)=dim(X^*)=s_{I,X^*}$. Then the above sum  is an algebraic integer by Proposition \ref{orthomatrix} $(3)$.
\end{proof}
Below we give some direct conclusions about slightly degenerate fusion categories.
\begin{lemm}\label{general2^ndsldege}Let $\C$ be a slightly degenerate fusion category and $FPdim(\C)=2^nd$, where d is an odd  square-free integer, then $\C$ is pointed if $1\leq n\leq2$ or $2^n\mid FPdim(\C_{pt})$.
\end{lemm}
\begin{proof} By Theorem \ref{slightstrictweakly}, $\C$ is integral if $n\leq2$, and $FPdim(X)^2|2^{n-1}d$ by Corollary \ref{sqdimension}, $\forall X\in\Q(\C)$. As $d$ is odd and square-free, hence $\C$ is pointed if $n\leq2$. If $2^n\mid FPdim(\C_{pt})$, then $\C_{ad}$ is braided and square-free by Theorem \ref{dimensionform}. If $\C'\cap\C_{ad}=Vec$, then $\C_{ad}$ is a modular fusion category; if $\C'\cap\C_{ad}=sVec$, and $\C_{ad}$ is slightly degenerate. Hence in both cases $\C_{ad}$ is pointed, and then $2^nd|FPdim(\C_{pt})$. That is, $\C$ is pointed.
\end{proof}
\begin{prop}\label{integralslightly2^nd}Let $\C$ be an integral non-pointed slightly degenerate fusion category of FP-dimension $2^nd$, where $d$ is an odd square-free integer, then $8|FPdim(\C_{pt})|2^{n-2}d$. Moreover, $\E'$ is group-theoretical, where $\E$ is a maximal Tannakian subcategory of $\C$.
\end{prop}
\begin{proof}If $\C$ is integral and non-pointed, so $n\geq3$ by Lemma \ref{general2^ndsldege}. By Corollary \ref{sqdimension} simple objects $X$ have FP-dimensions $2^i$, $0\leq i\leq t$ for some positive integer $t$. Then \begin{align*}
FPdim(\C)=FPdim(\C_{pt})+2^2 a_1+2^4 a_2+\cdots 2^{2t}a_t,
\end{align*}
where $a_j$ is the number of isomorphism classes of simple objects of FP-dimension $2^j$, and $2|a_j$ for all $1\leq j\leq t$ by Lemma \ref{slightMugercenter}, hence $8|FPdim(\C_{pt})$.

By Lemma \ref{general2^ndsldege} $2^n\nmid FPdim(\C_{pt})$. If $2^{n-1}\mid FPdim(\C_{pt})$,  then we can assume $FPdim(\C_{pt})=2^{n-1}$. Indeed if there exists odd prime $p|FPdim(\C_{pt})$, then $\C$ contains a modular fusion category $\A$ of FP-dimension $p$ since $d$ is square-free, then $\C\cong\A\boxtimes\A'$ and we replace $\C$ by $\A'$. By Theorem \ref{dimensionform} $FPdim(\C_{ad})=\left\{
 \begin{array}{ll}
2d, & \hbox{$\C'\cap\C_{ad}=Vec$,} \\
4d, & \hbox{$\C'\subseteq\C_{ad}$.}
\end{array}
\right.$. If $d=1$, then $\C$ is not integral if $\C'\cap\C_{ad}=Vec$, while if $\C'\cap\C_{ad}=sVec$, $\chi\subseteq X\otimes X^*$ for non-invertible simple object $X$ by Lemma \ref{slightMugercenter}, this is a contradiction. So, $d>1$. If $FPdim(\C_{ad})=2d$, then $\C_{ad}$ contains a unique non-trivial invertible simple object $g$, therefore $g\subseteq X\otimes X^*$ for any non-invertible simple object $X$ of $\C$. Meanwhile,  since $\C_{ad}$ is not pointed, for non-invertible object $X\in\Q(C_{ad})$, balancing equation  \cite[Proposition 8.13.8]{EGNO} and Proposition \ref{admugercenter} together mean that $\theta_g=1$, while this contradicts orthogonality of characters $s_{g,-}$ and $s_{I,-}$ in Proposition \ref{orthomatrix}. Then $FPdim(\C_{ad})=4d$ and $|U_\C|=2^{n-2}$, so we have equation
\begin{align*}
4d=FPdim(\C_h)=a_h+2^2(b_h)_1+2^4(b_h)_2+\cdots 2^{2t}(b_h)_t,
\end{align*}
where $a_h$ and $(b_h)_i$ are  the numbers of isomorphism classes of invertible objects and simple objects of FP-dimension $2^i$ contained in component $\C_h$ respectively, $ 1\leq i\leq t$, $\forall h\in U_\C$. Note that $\C'\subseteq\C_{ad}$, then Lemma \ref{slightMugercenter} implies $2|(b_h)_i$, $\forall 1\leq i\leq t$. Therefore $a_h\neq0$ $(\forall h\in U_\C)$, otherwise $2|d$, contradiction. Consequently $4|a_h$ $(\forall h\in U_\C)$, which then implies $FPdim(\C_{pt})\geq 2^n$, this is impossible.

As we have seen, if $\C$ is not pointed, then $2^{n-1}\nmid FPdim(\C_{pt})$ and $4\mid FPdim(\C_{ad})$. Hence, by computing $FPdim(\C_{ad})$, we see that $(\C_{ad})_{pt}$ always contains a non-trivial Tannakian subcategory, let $\E=Rep(G)$ be a maximal Tannakian subcategory of $\C$, by Corollary \ref{Tannakiandimen} $|G|=2^j$ for some positive integer $j$, then Remark \ref{Tannmugdeequiv} shows that  $\D:=(\E')_G$ is a slightly degenerate fusion category of FP-dimension $2^{n-2j}d$. If $\D$ is not pointed, then by induction on FP-dimension, $(\D_{ad})_{pt}$ contains a non-trivial Tannakian subcategory, however this contradicts \cite[Corollary 5.29]{DrGNO2}, since $\D$ is weakly anisotropic by \cite[Corollary 5.19]{DrGNO2}. As an equivariantization of pointed fusion category, then $\E'$ is group-theoretical \cite[Theorem 7.2]{NNW}.
\end{proof}
\begin{coro}\label{slightly2^nd}Let $\C$ be an integral  slightly degenerate fusion category of FP-dimension $2^nd$, if $FPdim(\C_{pt})=2^{n-2}m$, where $d$ is an odd square-free integer, $m|d$, then $\C$ is group-theoretical.
\end{coro}
\begin{proof}If $m\neq1$, there exists a non-degenerate pointed fusion subcategory $\A$ of FP-dimension $m$ as Proposition \ref{integralslightly2^nd}, so we can replace $\C$ by $\A'$, where $\A'$ is the centralizer of $\A$ in $\C$. Below we always assume $FPdim(\C_{pt})=2^{n-2}$.  By Theorem \ref{dimensionform} $FPdim(\C_{ad})=\left\{
 \begin{array}{ll}
4d, & \hbox{$\C'\cap\C_{ad}=Vec$,} \\
8d, & \hbox{$\C'\subseteq\C_{ad}$.}
\end{array}
\right.$. If $FPdim(\C_{ad})=4d$, then $(\C_{ad})_{pt}$ is a symmetric fusion category of FP-dimension $4$. If $(\C_{ad})_{pt}=\E= Rep(G)$ is Tannakian, then  $(\C_{ad})_G$ is a modular pointed fusion category of FP-dimension $d$; if $(\C_{ad})_{pt}$ is super-Tannakian, then it contains a maximal Tannakian subcategory $\E=Rep(G)$, so $(\C_{ad})_{G}$ is slightly degenerate pointed fusion category of FP-dimension $2d$ by Lemma \ref{general2^ndsldege}. In both cases, we have  $(\C_G)_{ad}\subseteq(\C_{ad})_G\subseteq\C_G$ by \cite[Proposition 4.30]{DrGNO2}. Therefore $\C_G$ is nilpotent then pointed by \cite[Corollary 5.3]{GN}, hence $\C$ is group-theoretical by
\cite[Theorem 7.2]{NNW}. The other case can be proved similarly.
\end{proof}
For an integral non-degenerate fusion category $\C$ of FP-dimension $2^nd$, where $d$ is an odd square-free integer, if $\C$ is group-theoretical, then it is nilpotent.  Indeed, by \cite[Corollary 4.14]{DrGNO1} there exists a symmetric fusion category $\E\subseteq\C$ such that $\E'$ is nilpotent, since $FPdim(\E)$ is a prime power of $2$ \cite[Proposition 4.56]{DrGNO2}, \cite[Theorem 6.10]{DrGNO1} shows that $\C$ contains a modular fusion category $\A$ of FP-dimension $d$, so by \cite[Theorem 3.13]{DrGNO2} $\C\cong\A\boxtimes\A'$ is nilpotent. Now we are ready to improve result of \cite{DLD}.

\begin{prop}\label{integralmodular2^nd}
Integral non-degenerate fusion category $\C$ of FP-dimension $2^nd$ are nilpotent, if $n\leq6$.
\end{prop}
\begin{proof}
Assume $\C$ is not pointed, otherwise there is nothing  to prove. Since $\C$ is integral, FP-dimensions of simple objects of $\C$ are powers of $2$, as in Proposition \ref{integralslightly2^nd}, we can assume $4|FPdim(\C_{pt})|2^{n-2}$ and $4|FPdim((\C_{ad})_{pt})$, so $n\geq4$. Indeed, we  show if $\C$ is not nilpotent then $4|FPdim(\C_{pt})|2^{n-4}$.

If $FPdim(\C_{pt})=2^{n-2}$,  as in Corollary \ref{slightly2^nd}, $(\C_{ad})_{pt}$ is symmetric, then we can show  $\C$ is group-theoretical, hence $\C$ is nilpotent. If $FPdim(\C_{pt})=2^{n-3}$ and $FPdim(\C_{ad})=2^3d$, consider $FPdim((\C_{ad})_{pt})$: if $FPdim((\C_{ad})_{pt})=8$, then $\C$ is nilpotent as in Corollary \ref{slightly2^nd}; if $FPdim((\C_{ad})_{pt})=4$, we show $(\C_{ad})_{pt}$  is Tannakian. We have equation
\begin{align*}
8d=FPdim(\C_{ad})=4+2^2 a_1+2^4 a_2+\cdots 2^{2t}a_t,
\end{align*}
where $a_j$ is the number of isomorphism classes of simple objects of FP-dimension $2^j$, $1\leq j\leq t$, therefore $a_1$ is odd. Let $G(\C_{ad})$ act on subset $A$, where $A:={\{X\in\Q(\C_{ad})| FPdim(X)=2}\}$, then there exists a fixed simple object $X$, which implies that
$X\otimes X^*=\oplus_{ g\in G(\C_{ad})}g$, balancing equation \cite[Proposition 8.13.8]{EGNO} means $(\C_{ad})_{pt}$ is Tannakian. Then $\C$ is group-theoretical as  in Corollary \ref{slightly2^nd}.

Hence for $n\leq5$, $\C$ is  nilpotent. If $n=6$, then we can assume that $\C_{pt}=(\C_{ad})_{pt}=Rep(\Z_2\times\Z_2)$ is Tannakian. Let $G(\C)=\langle g\rangle\times\langle h\rangle$, and consider  Tannakian subcategories $\E_i$ generated by $g$, $h$ and $gh$ respectively, $1\leq i\leq3$, then their M\"{u}ger centers are group-theoretical, since they are equivalent to $(\C^0_{\Z_2})^{\text{rev}}$ by \cite[Proposition 4.56]{DrGNO2}, which  is nilpotent then pointed  modular fusion category of FP-dimension $16d$ by \cite[Lemma 4.11]{DongNa}, therefore simple objects of these M\"{u}ger centers $\E'_i$ have FP-dimension $1$ or $2$ by \cite[Proposition 4.26]{DrGNO2}. While this implies $FPdim(X)\leq2$, $\forall X\in\Q(\C)$, since $\E_i\cap \E_j=\C_{ad}$ for $1\leq i\neq j\leq 3$, by computing FP-dimension of $\C$, we obtain that $\Q(\C)=\cup_{i=1}^3\Q(\E_i)$, then $\C$ is group-theoretical and nilpotent by \cite[Theorem 1.2]{Na2}.
\end{proof}

\begin{remk} It is interesting to see whether Proposition \ref{integralmodular2^nd} is true for all integers $n$, we leave it for future consideration.
\end{remk}

\begin{lemm}\label{general2p^nd4p^nd}Let $\C$ be a slightly degenerate fusion category with $FPdim(\C)=2p^nd$ or $4p^nd$, where $p$ is an odd prime and $d$ is an odd square-free integer, $(p,d)=1$. If $\C$ is not pointed, then $p^{n-1}\nmid FPdim(\C_{pt})$ and $(\C_{ad})_{pt}$ contains a Tannakian subcategory $\E$ of FP-dimension $p^j$, $j\geq2$.
\end{lemm}
\begin{proof} We only show the case $FPdim(\C)=2p^nd$, the other is same. By Proposition \ref{striweakintchar} $\C$ is integral; since $\C$ is not pointed, non-invertible simple objects of $\C$ have FP-dimensions of $p^j$ by Corollary \ref{sqdimension} ($j\geq1$), therefore $p|FPdim((\C_{ad})_{pt})$ and  $p^n\nmid FPdim(\C_{pt})$ by Theorem \ref{dimensionform}.

If $p^{n-1}\mid FPdim(\C_{pt})$, assume $FPdim(\C_{pt})=2p^{n-1}m$, $d=ms$,  by Theorem \ref{dimensionform},
 $FPdim(\C_{ad})=\left\{
 \begin{array}{ll}
ps, & \hbox{$\C'\cap\C_{ad}=Vec$,} \\
2ps, & \hbox{$\C'\subseteq\C_{ad}$.}
\end{array}
\right.$ and  $|U_\C|=\left\{
 \begin{array}{ll}
2p^{n-1}m, & \hbox{$\C'\cap\C_{ad}=Vec$,} \\
p^{n-1}m, & \hbox{$\C'\subseteq\C_{ad}$.}
\end{array}
\right.$. Since $FPdim(\C_{ad})=FPdim(\C_g)=a_g+p^2b_g$ $(\forall g\in U_\C)$, where $a_g$ is the number of isomorphism classes of invertible objects contained in component $\C_g$, $a_g\neq0$, otherwise $p\mid d$, so $p|a_g$ $(\forall g\in U_\C)$. However then $FPdim(\C_{pt})\geq p|U_\C|>FPdim(\C_{pt})$ as $p$ is an odd prime, this is a contradiction, hence $p^{n-1}\nmid FPdim(\C_{pt})$ and $p^2\mid FPdim(\C_{ad})$. By calculating FP-dimension of $\C_{ad}$, we obtain $p^2|FPdim((\C_{ad})_{pt})$; by Proposition \ref{admugercenter} $(\C_{ad})_{pt}$ is symmetric, then $(\C_{ad})_{pt}$ contains a  Tannakian subcategory $\E$ of FP-dimension $p^j$ by \cite[Corollary 2.50]{DrGNO2} with $j\geq2$.
\end{proof}
\begin{theo}\label{2p^n4p^n}Slightly degenerate fusion categories of FP-dimension $2p^n$ and $4p^n$ are nilpotent where $p$ is an odd prime and $n$ is a positive integer.
\end{theo}
\begin{proof}Let $\C$ be a slightly degenerate fusion category of FP-dimension $2p^n$. If $\C'\cap\C_{ad}=Vec$, then it follows from Proposition \ref{admugercenter} and Theorem \ref{dimensionform} that $FPdim(\C_{ad})=p^t$, $\C$ is nilpotent by \cite[Theorem 8.28]{ENO1}. Assume $\C'\subseteq\C_{ad}$ and $\C$ is not pointed, then $FPdim(\C_{pt})=2p^m$ and $FPdim(\C_{ad})=2p^{n-m}$ by Theorem \ref{dimensionform}, and $2\leq m\leq n-2$, $(\C_{ad})_{pt}$ contains a maximal Tannakian subcategory $Rep(G)$ and $|G|=p^j$, $j\geq2$ by Lemma \ref{general2p^nd4p^nd}. Then $(\C_{ad})_G$ is a slightly degenerate fusion category of FP-dimension $2p^{n-m-j}$ by Remark \ref{Tannmugdeequiv}, by induction $(\C_{ad})_G$ is nilpotent. Hence $(\C_{ad})_G\cong sVec\boxtimes\D$, where $\D$ is a modular fusion category of FP-dimension $p^{n-m-j}$, and $\C_{ad}\cong sVec^G\vee\D^G$. Since $G$ acts on $sVec$ trivially, and $FPdim(\D^G)=p^{n-m}$,   \cite[Theorem 8.28]{ENO1} and \cite[Lemma 2.2]{DongNa} together imply  that $\C_{ad}$ is nilpotent, then $\C$ is nilpotent.

Let $FPdim(\C)=4p^n$ and assume $\C$ is not pointed. Then $FPdim(\C_{pt})=2^kp^m$, with $k=1,2$ and $2\leq m\leq n-2$ by Lemma \ref{general2p^nd4p^nd}. If $k=2$, by Theorem  \ref{dimensionform} $FPdim(\C_{ad})=\left\{
  \begin{array}{ll}
 p^{n-m}, & \hbox{$\C'\cap\C_{ad}=Vec$,} \\
 2p^{n-m}, & \hbox{$\C'\cap\C_{ad}=sVec$.}
\end{array}
 \right.$, then $\C_{ad}$ is nilpotent as the case $FPdim(\C)=2p^n$; If $k=1$,
then $FPdim(\C_{ad})=\left\{
  \begin{array}{ll}
 2p^{n-m}, & \hbox{$\C'\cap\C_{ad}=Vec$,} \\
 4p^{n-m}, & \hbox{$\C'\cap\C_{ad}=sVec$.}
\end{array}
 \right.$.  Then $(\C_{ad})_G$ is braided and $((\C_{ad})_G)'\subseteq sVec$, where $Rep(G)$ is a maximal Tannakian subcategory of $(\C_{ad})_{pt}$, $(\C_{ad})_G$ is nilpotent by induction if $4|FPdim(\C_{ad})$ and by \cite[Theorem 4.7]{DongNa} if $4\nmid FPdim(\C_{ad})$, then $\C_{ad}$ is nilpotent by similar argument.
\end{proof}
The following corollary is an analog  of \cite[Theorem 4.7]{DongNa}.
\begin{coro}\label{slightly2p^nd4p^nd}Slightly degenerate fusion categories  of FP-dimension $2p^nd$  and $4p^nd$ are nilpotent, where p is an odd prime and d is an odd square-free integer such that $(p,d)=1$, $n$ is a non-negative integer. In particular, they are  group-theoretical.
\end{coro}
\begin{proof}Let $\C$ be a slightly degenerate fusion categories  of FP-dimension $2p^nd$ with $d>1$($FPdim(\C)=4p^nd$ is similar), then $\C$ is integral by Theorem \ref{slightstrictweakly}. When $n\leq3$, it follows from Lemma \ref{general2p^nd4p^nd} that $\C$ is pointed, definitely $\C$ is nilpotent.

There is nothing to prove if  $\C$ is pointed.
Below we assume $\C$ is not pointed, and necessarily $n\geq4$. As in Theorem \ref{slightstrictweakly}, if there exists an odd prime $q\neq p$ and $q|FPdim(\C_{pt})$,  then $\C_{pt}$ contains  a modular fusion category $\B$ of FP-dimension $q$, $\C\cong\B\boxtimes \B'$ by \cite[Theorem 3.13]{DrGNO2} and we can replace $\C$ with $\B'$. Therefore, by Theorem \ref{dimensionform} we further assume $FPdim(\C_{pt})=2p^m$, $FPdim(\C_{ad})=\left\{
 \begin{array}{ll}
p^{n-m}d, & \hbox{$\C'\cap\C_{ad}=Vec$,} \\
2p^{n-m}d, & \hbox{$\C'\subseteq\C_{ad}$.}
\end{array}
\right.$, $n-2\geq m\geq2$. We proceed by induction on $n$.

Let $\E=Rep(G)$ be a maximal Tannakian subcategory of $(\C_{ad})_{pt}$, $|G|=p^j$ with $j\geq2$ by Lemma \ref{general2p^nd4p^nd}, then $(\C_{ad})_G$ is slightly degenerate of FP-dimension $2p^{n-m-j}d$ Remark \ref{Tannmugdeequiv}. Here we assume $\C'\subseteq\C_{ad}$, if not, then $(\C_{ad})_G$ is a modular fusion category of FP-dimension $p^{n-m}d$, then $(\C_{ad})_G$ is nilpotent by \cite[Theorem 4.7]{DongNa}. By induction, $\D:=(\C_{ad})_G$ is nilpotent, by \cite[Theorem 6.10]{DrGNO1} $\D\cong sVec\boxtimes\D_{q_1}\boxtimes\cdots\boxtimes\D_{q_r}\boxtimes \D_p$ as braided fusion categories, where $d=q_i\cdots q_r$, $\D_{q_i}$ is a modular fusion category of FP-dimension $q_i$, $\D_p$ is a modular pointed fusion category with $FPdim(\D_p)=p^{n-m-j}$, $q_i$ are odd primes, $1\leq i\leq r$. Then $\C_{ad}=\D^G=(sVec)^G\vee(\D_{q_1})^G\vee\cdots\vee(\D_{q_r})^G\vee(\D_p)^G$.

Note that $G$ acts on $sVec$ trivially and $(\D_p)^G$ is a fusion category of FP-dimension $p^{n-m}$, so they are nilpotent fusion subcategories. We show $(\D_{q_i})^G$ are nilpotent for all $1\leq i\leq r$, then by \cite[Lemma 2.2]{DongNa} $\C_{ad}$ is nilpotent, which means $\C$ is nilpotent. By
\cite[Proposition A.6]{DrGNO2}, the action of $G$  on a metric group of odd order $q_i$ while preserving the quadratic form can never be non-trivial. Moreover, $\forall q\geq0$,  since $(|G|,q_i)=1$, $H^q(G,\Z_{q_i})=0$ for all $1\leq i\leq r$.  Therefore $(\D_{q_i})^G\cong \D_{q_i}\boxtimes Rep(G)$, then $q_i|FPdim(\C_{pt})$, and consequently $d=1$ as $FPdim(\C_{pt})=2p^m$. Then by Theorem \ref{2p^n4p^n}, $\C$ is a nilpotent fusion category.

Nilpotency of $\C$  means that $\C$ is  equivalent to Deligne tensor products of braided fusion categories of prime power FP-dimensions  \cite[Theorem 6.12]{DrGNO1}. Then $\C$ is group-theoretical by \cite[Theorem 6.10]{DrGNO1}.
\end{proof}

Therefore, like classification of modular fusion category of FP-dimension $2p^nd$
\cite[Theorem 4.7]{DongNa}, the classification of slightly degenerate fusion category $\C$ of FP-dimension $2p^nd$ or $4p^nd$ is reduced to classification of modular fusion categories of prime power FP-dimensions. Particularly, if $n\leq4$, then $\C$ is pointed by \cite[Lemma 4.11]{DongNa}.

\begin{coro}Integral braided fusion categories of FP-dimension $p^nd$ are group-theoretical if $n\leq4$, where $p$ is a prime, $d$ is a square-free integer and $(p,d)=1$.
\end{coro}
\begin{proof}Let $\C$ be a braided fusion category of  FP-dimension $p^nd$ with $n\leq4$. Consider the M\"{u}ger center $\C'$ of $\C$: if $\C'=Vec$, then $\C$ is modular and pointed by
\cite[Corollary 4.13]{DongNa}, and Proposition \ref{integralmodular2^nd}; if $\C'=sVec$, then $\C$ is pointed by Proposition \ref{integralslightly2^nd} and Corollary \ref{slightly2p^nd4p^nd}; if $\C'\cong Rep(G)$ is a non-trivial Tannakian category, then $\C\cong\D^G$ is group-theoretical by \cite[Theorem 7.2]{NNW}, since $\D$ is modular and by induction $\D$ is pointed; if $\C'=Rep(G,u)$, then it contains a maximal Tannakian subcategory $\E\cong Rep(N)$ and $|N|=\frac{1}{2}FPdim(\C')$, by Remark \ref{Tannmugdeequiv} $\C_N$ is a slightly degenerate fusion category, which is pointed by induction on FP-dimension of $\C_N$, then $\C$ is group-theoretical.
\end{proof}
The following proposition generalizes result of  \cite[Proposition 5.3]{Na2}.
\begin{prop}\label{simplep^mq^nsolv}Let $p,q$ be primes, and $\C$ be a weakly integral braided fusion category such that integral simple objects of $\C$ have FP-dimensions $p^mq^n$, where $m,n$ are nonnegative integers. Then $\C$ is weakly group-theoretical.
\end{prop}
\begin{proof}Let $Rep(G)\subseteq\C'$ be  a maximal Tannakian subcategory, then $\C\cong\D^G$, where  $\D'\subseteq sVec$.   The argument of \cite[Proposition 5.3]{Na2} shows that $G$ is a solvable group. Therefore, it is sufficient to show   $\D$ is weakly group-theoretical by \cite[Proposition 4.1]{ENO2}. If $\D$ is non-degenerate, then $\D$ is solvable by \cite[Proposition 5.3]{Na2}.

Assume that $\D$ is slightly degenerate.  When $\D$ is   nilpotent, the proposition  follows directly from \cite[Theorem 6.10]{DrGNO1}. We also assume that $\D$ does not contain non-degenerate fusion subcategories. In fact, if not,  let $\A$ be a non-degenerate fusion subcategory of $\D$, then $\D\cong\A\boxtimes\A'$ by \cite[Theorem 3.13]{DrGNO2}, then $\A$, $\A'$ and $\D$ are weakly group-theoretical by induction. Moreover, if $\D$ contains a non-trivial Tannakian subcategory $Rep(N)$, then $\D_N^0$ is again slightly degenerate and $FPdim(\D_N^0)=\frac{FPdim(\D)}{|N|^2}$ by Corollary \ref{Tannakiandimen}, by induction $\D_N^0$ is weakly group-theoretical, so is $\D$ \cite[Proposition 4.1]{ENO2}. Therefore, we only need to prove that $\D$ always contains a non-trivial Tannakian subcategory. On the contrary, let  $\D$ be an anisotropic slightly degenerate fusion category, that is, the only Tannakian subcategory of $\D$ is equivalent to $Vec$.

If $\D=\D_{ad}$, then $\D$ is integral and $\D'=\D_{pt}=sVec$ by Proposition \ref{admugercenter}.
Particularly, $\D$ does not contain simple objects of odd prime power FP-dimensions by \cite[Proposition 7.4]{ENO2}. If $p,q$ are odd primes, then $pq|FPdim(X)$ for all $X\in\Q(\D)$, then $pq|FPdim(\D_{pt})$, this is impossible. Then  non-invertible simple objects of $\D$ all have even FP-dimensions by assumption. Hence Corollary \ref{sqdimension} implies that $8|FPdim(\D)$. While
$$FPdim(\D)=FPdim(\D_{pt})+\sum\limits_{\substack{X\in\Q(\D)\\ FPdim(X)>1}}FPdim(X)^2,$$
then Lemma \ref{slightMugercenter} implies that $8|FPdim(\D_{pt})$, this is a contradiction. Therefore, $\D$ contains a non-trivial Tannakian subcategory $Rep(N)$. By induction, $\D_N^0$ is weakly group-theoretical, so is $\D$  by \cite[Proposition 4.1]{ENO2}.

If $\D_{ad}$ is a proper subcategory of $\D$, then it follows from \cite[Proposition 3.29]{DrGNO2} that $(\D_{ad})_{pt}$ is a symmetric fusion category, hence $(\D_{ad})_{pt}\subseteq sVec$, as $\D$ is anisotropic. If $(\D_{ad})_{pt}= Vec$, then $\D_{ad}$ is non-degenerate. Consequently,  $\D_{ad}=Vec$ since we assume that $\D$ does not contain   non-degenerate fusion subcategories, which implies that $\D$ is pointed, this is impossible. Then $\D_{ad}$ is an integral slightly degenerate fusion category, previously argument implies that $\D_{ad}$ contains a non-trivial Tannakian subcategory.
\end{proof}
By Corollary \ref{sqdimension}, Proposition \ref{simplep^mq^nsolv} and \cite[Theorem 5.1]{Na2}, slightly degenerate fusion categories of FP-dimensions $p^mq^nd$ are always solvable, where $p,q$ are primes, $d$ is a square-free integer such that $(pq,d)=1$.
\begin{remk}Let $2<p<q<r$ be primes. Using Proposition $\ref{orthomatrix}$, one can show that  slightly degenerate fusion categories $\C$ of FP-dimension $2p^2q^2r^2$  contain a non-invertible simple objects of prime power FP-dimension, then $\C$ contains a non-trivial Tannakian subcategory by \cite[\text{Proposition $7.4$}]{ENO2}, hence  $\C$ is  weakly group-theoretical. Then $\C$ is solvable by \cite[Theorem 5.1]{Na2}
\end{remk}

\section{Slightly degenerate  fusion categories of particular dimensions}\label{section4}
In this section we study slightly degenerate generalized Tambara-Yamagami fusion categories and classify slightly degenerate fusion categories of FP-dimension $2^n$ ($n\leq5$) and $8d$, where $d$ is a square-free integer. We begin with classification of slightly degenerate generalized Tambara-Yamagami fusion category.

Let $\C$ be a generalized Tambara-Yamagami fusion category, i.e. for non-invertible simple objects  $X,Y\in \mathcal{O}(\C)$, $X\otimes Y\in\C_{pt}$, generalized Tambara-Yamagami fusion categories were classified up to tensor equivalence in \cite{Lip}. By \cite[Lemma 5.1]{Na} there exists a normal subgroup $\Gamma\subseteq G(\C)$ such that for any non-invertible simple object $X$, $X\otimes X^*=\oplus_{g\in\Gamma}g$. Therefore $FPdim(X)\in {\{1,\sqrt{|\Gamma|}}\}$ $(X\in\Q(\C))$ and $\C_{ad}\cong Vec^\omega_\Gamma$ as fusion category. By definition, $\C$ is nilpotent but not pointed.
\begin{example}Let $\C=\mathcal{TY}(\Gamma,\tau,\mu)$ be a Tambara-Yamagami fusion category, where $\Gamma$ is a  finite abelian group, $\tau:\Gamma\times\Gamma\to k^*$ is a symmetric non-degenerate bicharacter, $\mu\in k^*$ such that $|\Gamma|\mu^2=1$. Let $X$ be the unique non-invertible simple object of $\C$, then $X\otimes X=\oplus_{g\in\Gamma}g$ and  $FPdim(\C)=2|\Gamma|$. Tambara-Yamagami fusion categories are classified up to tensor equivalence in \cite{TY}.  By \cite[\text{Theorem} 1.2]{Si}, $\C$ admits a braiding structure if and only if $\Gamma$ is an elementary abelian $2$-group.
\end{example}
\begin{example}\label{ising}Let $\I=\mathcal{TY}(\Z_2,\tau,\mu)$ be a Tambara-Yamagami fusion category, i.e. $\I$ is an Ising category.  Then $\I$ has a braiding structure. In fact, All braided Ising categories are modular, and braided Ising categories  are classified in \cite[Appendix B]{DrGNO2}.

Let  $N$ be  a positive integer. In \cite{DNS}, the authors defined $N$-Ising categories $\I_N$, which is a non-pointed braided $\Z_{2^N}$-extension of a pointed fusion category $Vec_{\Z_2}$, so $\I_N$ is a generalized Tambara-Yamagami fusion category. And all proper fusion subcategories of $\I_N$ are pointed  by \cite[Theorem 4.7]{DNS}. In particular, any non-invertible simple object $X\in\I_N$ generates $N$-Ising category $\I_N$. For $N=1$, $\I_1$ is exactly an Ising category.
\end{example}

If $\C$ is a slightly degenerate generalized Tambara-Yamagami fusion category, by
\cite[Theorem 6.12]{DrGNO1} nilpotency of $\C$ means that $\C\cong \A\boxtimes\B_1\boxtimes\cdots\boxtimes\B_s$ as braided fusion category, where $FPdim(\A)=2^m$, and $FPdim(\B_i)=p_i^{n_i}$, $p_i$ are odd primes, as $\C'=sVec$, so $\A$ is slightly degenerate, all $\B_i$  are modular, $1\leq i\leq n$. Meanwhile \cite[Theorem 5.4]{Na} shows that $\B_i$ are pointed ($1\leq i\leq s$), therefore $\C\cong \A\boxtimes \D$, where $\D$ is a pointed modular fusion category. Hence the classification of slightly degenerate generalized Tambara-Yamagami fusion category $\C$ is reduced to when $FPdim(\C)=2^m$, obviously we have $m\geq3$.

The following theorem is an application of Proposition \ref{orthomatrix}, and will be useful in the classifications throughout this section.
\begin{theo}\label{subgpofadjoint}There exists no  slightly degenerate  fusion category $\C$, which satisfies the following condition: there exists a subgroup $\Gamma\subseteq G(\C)$ such that for any non-invertible simple object $X$, $\forall g\in\Gamma$, $g\otimes X\cong X$, and $|\Gamma|\geq3$.
\end{theo}
\begin{proof} Assume there exists such a fusion category $\C$, by Lemma \ref {slightMugercenter}, $\chi\notin\Gamma$. Let us choose a partition of $\Q(\C)$ such that $\Gamma\subseteq\Pi_0$. For any non-trivial invertible simple object $g\in\Gamma$, by Proposition \ref{orthomatrix}, orthogonality of characters $s_{g,-}$ and $s_{I,-}$ means $\Sigma_{V\in\Pi_0}s_{I,V}s_{V,g}=0$. Using the balancing equation \cite[Proposition 8.13.8]{EGNO}, we have $s_{V,g}=FPdim(V)\theta^{-1}_g$; by Proposition \ref{admugercenter} $1=s_{g,h}=\theta^{-1}_g\theta^{-1}_h\theta_{gh}$ for all $h\in\C_{pt}$.
 Then $0=\frac{1}{2}FPdim(\C_{pt})+
\theta^{-1}_g\sum\limits_{\substack{V\in\Pi_0,\\ V\notin\Q(\C_{pt})}}FPdim(V)^2$,
 hence $\theta_g=-1$ as it is a root of unity. However, as $|\Gamma|\geq3$, there exist non-trivial invertible objects $g,h\in\Gamma$ and $g\neq h^{-1}$,  then $-1=\theta_{gh}=\theta_g\theta_h=1$, this is a  contradiction.
\end{proof}

\begin{prop}\label{slightdim8}Let  $\C$ be a slightly degenerate fusion category of FP-dimension $8$, then  $\C\cong sVec\boxtimes \D$, where $\D$ is non-degenerate fusion category of FP-dimension $4$.
\end{prop}
\begin{proof}When $\C$ is integral, this is a direct result of Lemma \ref{integralslightly2^n} below.
If $\C$ is strictly weakly integral, then  $\C_{int}=\C_{pt}$. Nilpotency of $\C$ implies $FPdim(X)^2|FPdim(\C_{ad})$ ($X\in\Q(\C)$) by \cite[Corollary 5.3]{GN}, therefore $FPdim(X)\in{\{1,\sqrt{2}}\}$. Then it follows from Lemma \ref{slightMugercenter} and partition of $\Q(\C)$ in section \ref{section3} that simple objects of FP-dimension $\sqrt{2}$ are self-dual, since there are exactly two simple objects of FP-dimension $\sqrt{2}$. Therefore, $\C\cong sVec\boxtimes \I$ by \cite[Theorem 3.13]{DrGNO2}, where $\I$ is an Ising category.
\end{proof}

\begin{coro}\label{generalizedTYstru}Let $\C$ be a slightly degenerate generalized Tambara-Yamagami fusion category with $FPdim(\C)=2^m$, $m\geq3$, then $\C\cong sVec\boxtimes\I\boxtimes\D$, where  $\D$ is a non-degenerate pointed fusion category.
\end{coro}
\begin{proof} By definition and Theorem \ref{subgpofadjoint}, $\Gamma\cong\Z_2$. Let $\Gamma=\langle g\rangle$, then the proof of Theorem \ref{subgpofadjoint} implies $\theta_g=-1$, that is, $\C_{ad}\cong sVec$ as symmetric fusion categories.

Note that M\"{u}ger center  of $\C_{pt}$  is $\C_{ad}\vee \C'$ by Proposition \ref{admugercenter}, which has FP-dimension $4$. So, \cite[Lemma 5.1]{DNS} says that $\C_{pt}\cong \C'\boxtimes \C_0$ for some pointed braided fusion category $\C_0$. While $\C_{ad}\cong sVec\subseteq\C_0$ also centralizes $\C_{0}$, again we obtain that $\C_0\cong \C_{ad}\boxtimes\D$ for another braided pointed fusion category $\D$. Consequently, $\C_{pt}\cong \C'\boxtimes\C_{ad}\boxtimes\D$. Then  $\D$ must be a non-degenerate fusion category. Since M\"{u}ger of $\C_{pt}$ is exactly $\C'\boxtimes\C_{ad}\boxtimes\D'$ by definition, where $\D'$ is  the M\"{u}ger center of $\D$, $FPdim(\C'\boxtimes\C_{ad}\boxtimes\D')=4$ if and only if $\D'=Vec$ if and only if $\D$ is non-degenerate.

Therefore, we have a  braided fusion category equivalence $\C\cong\D\boxtimes\A$ by \cite[Theorem 3.13]{DrGNO2}, where $\A$ is a slightly degenerate fusion category of FP-dimension $8$. Since $\C$ is strictly weakly integral, $\A$ must be strictly weakly integral. Proposition \ref{slightdim8} says that $\A\cong\I\boxtimes sVec$. Therefore, $\C\cong sVec\boxtimes\I\boxtimes\D$ as required.
\end{proof}

\begin{remk}\label{N-Ising}
Note that Corollary \ref{generalizedTYstru} also shows that $N$-Ising categories $\I_N$ can never be slightly degenerate. Hence, together with \cite[Lemma 4.12]{DNS}, we obtain that M\"{u}ger center of $\I_N$ does not contains a symmetric fusion category $sVec$.
\end{remk}
Next we give classifications of slightly degenerate fusion categories of FP-dimensions $2^n$ $(n\leq5)$. There is nothing to prove when $n\leq2$, since $\C$ is pointed. We begin with integral slightly degenerate fusion categories.
\begin{remk}\label{premodular8}Let $\C$ pre-modular fusion category of FP-dimension $8$, then $\C$ is pointed or $\mathcal{TY}(\Z_2\times\Z_2,\tau,\mu)$ if $\C$ is integral by \cite[Theorem 1.2]{Si}. Meanwhile, $\C$ is a generalized Tambara-Yamagami fusion category when $\C$ is strictly weakly integral, and simple objects have FP-dimensions $1$ or $\sqrt{2}$,  if $\C$ does not contain a self-dual non-invertible simple object, then $\C\cong\I_2$ and $\C'\ncong sVec$ by Remark \ref{N-Ising} and \cite[Theorem 5.5]{DNS}.
\end{remk}
\begin{lemm} \label{integralslightly2^n}Let $\C$ be an integral  slightly degenerate fusion category of FP-dimension $2^n$, if $\C$ is not pointed, then $n\geq6$.
\end{lemm}
\begin{proof}Let $\C$ be an integral slightly degenerate  fusion category of FP-dimension $2^n$, by Lemma \ref{integralslightly2^nd} it suffices to show that if $n=5$ then $\C$ is pointed.
On the contrary, assume that $\C$ is not pointed, then $FPdim(\C_{pt})=8$ and $FPdim(\C_{ad})=8$ by Lemma \ref{integralslightly2^nd}, so $\C_{ad}=\C_{pt}$, Proposition \ref{admugercenter} implies $\C_{ad}$ is symmetric.  Also nilpotency of $\C$ implies $FPdim(X)\in{\{1,2}\}$ ($X\in\Q(\C)$) by \cite[Corollary 5.3]{GN}.

Note that group $U_\C$ is isomorphic to either $\Z_4$ or $\Z_2\times\Z_2$, so $\C$ contains a non-pointed fusion category $\A\supseteq\C_{pt}$ of FP-dimension $16$, Lemma \ref{slightMugercenter} means there exists a self-dual non-invertible simple object $X\in\Q(\A)$, then $X\otimes X=I\oplus g\oplus h\oplus gh$ and $\Z_2\times\Z_2=\langle g\rangle\times\langle h\rangle$, since $X$ generates a braided Tambara-Yamagami fusion category of FP-dimension $8$ \cite[Theorem 1.2]{Si}. By Proposition \ref{admugercenter} $g,h$ generate a symmetric fusion subcategory $\B$, so $\B$ contains a non-trivial Tannakian subcategory $\E$, without loss of generality, assume $\theta_g=1$.

Let $\Pi_0={\{I,g,h,gh,X,U,V}\}$, where $X,U,V$ are non-isomorphic simple objects of FP-dimension $2$. Then by Proposition \ref{orthomatrix}, orthogonality of  characters $s_{g,-}$ and $s_{I,-}$ implies $4+s_{g,U}+s_{g,V}=0$, then $\theta_U=-\theta_{g\otimes U}$ and $\theta_V=-\theta_{g\otimes V}$; orthogonality of characters $s_{I,-}$ and $s_{h,-}$ implies $1+\theta_h^{-1}+\theta_h^{-1}\theta_U^{-1}\theta_{h\otimes U}+\theta_h^{-1}\theta_V^{-1}\theta_{h\otimes V}=0$; orthogonality of characters $s_{g,-}$ and $s_{h,-}$ implies $1+\theta_h^{-1}-\theta_h^{-1}\theta_U^{-1}\theta_{h\otimes U}-\theta_h^{-1}\theta_V^{-1}\theta_{h\otimes V}=0$, then $\theta_h=-1=\theta_{gh}$ and $s_{X,X}=0$; finally orthogonality of characters $s_{g,-}$ and $s_{X,-}$ implies $4-s_{g,U}-s_{g,V}=0$, while this contradicts to equality $4+s_{g,U}+s_{g,V}=0$.
\end{proof}

Hence Lemma \ref{integralslightly2^nd}, Lemma \ref{general2p^nd4p^nd} and Lemma \ref{integralslightly2^n} show that integral slightly degenerate fusion categories $\C$ of FP-dimension less than $64$ are pointed, then $\C\cong sVec\boxtimes\D$ by Proposition \ref{slighpointed}, where $\D$ is a pointed non-degenerate fusion category.

\begin{prop}\label{slight16}Let $\C$ be a  slightly degenerate fusion category  of FP-dimension $16$, then  $\C$ is pointed  or $\C\cong sVec\boxtimes\I \boxtimes\D$, where $\D$ is a non-degenerate fusion category.
\end{prop}
\begin{proof}
If $\C$ is integral, then it is pointed by Lemma \ref{integralslightly2^n}. Now assume $\C$ is strictly weakly integral, obviously $4|FPdim(\C_{pt})$, then  $FPdim(\C_{pt})=8$ and $FPdim(\C_{ad})|4$, $\C_{int}=\C_{pt}$. Indeed, if $FPdim(\C_{pt})=4$, then
$FPdim(\C_{ad})=\left\{
\begin{array}{ll}
4, & \hbox{$\C'\cap\C_{ad}=Vec$,} \\
8, & \hbox{$\C'\cap\C_{ad}=sVec$.}
\end{array}
\right.$ by Theorem \ref{dimensionform}, while in the first case $FPdim(\C_{pt})=8$ and in the second case this contradicts to Lemma \ref{slightMugercenter}. Since $\C$ is nilpotent, $FPdim(X)^2|4$ ( $\forall X\in\Q(\C)$) by \cite[Corollary 5.3]{GN}. Note that $\C_{pt}=\C_{int}$, $FPdim(X)\in{\{1,\sqrt{2}}\}$, for arbitrary simple object $X\in\Q(\C)$, then $\C$ is a slightly degenerate  generalized Tambara-Yamagami fusion category. Then Corollary \ref{generalizedTYstru}  implies the classification.
\end{proof}
\begin{lemm}\label{nondegen16}Let $\C$ be a non-degenerate fusion category of FP-dimension $16$, then $\C$ is pointed or $\C\cong \I\boxtimes\D$, where $\D$ is equivalent to an Ising category or $\D$ is pointed.
\end{lemm}
\begin{proof}If $\C$ is integral, then $\C$ is pointed, as slightly degenerate fusion category $sVec\boxtimes\C$ is pointed by Lemma \ref{integralslightly2^n}. If $\C$ is strictly weakly integral, then $FPdim(\C_{pt})|8$. If $FPdim(\C_{pt})=2$, then $FPdim(\C_{ad})=8$ by \cite[Corollary 6.8]{GN}, this is impossible, since non-invertible simple objects of $\C_{ad}$ have FP-dimension $2$ by \cite[Corollary 5.3]{GN}.

When $FPdim(\C_{pt})=8$, $\C$ is a non-degenerate generalized Tambara-Yamagami fusion category, so $\C\cong\I\boxtimes\D$ by \cite[Theorem 5.4]{Na}, where $\D$ is a pointed non-degenerate fusion category. If $FPdim(\C_{ad})=4$, then $\C_{ad}=\C_{pt}$ is a symmetric fusion category by \cite[Corollary 6.8]{GN}. Since $\C$ is strictly weakly integral, $\C_{ad}$ is not Tannakian, otherwise $\C$ is braided equivalent to Drinfeld center of a pointed fusion category by \cite[Theorem 4.64]{DrGNO2}. Therefore, $\C_{ad}$ contains a subcategory $\A\cong sVec$, then centralizer $\A'$ of $\A$ is a slightly degenerate fusion category of FP-dimension $8$. By Proposition \ref{slightdim8}, $\C$ contains an Ising category $\I$, then $\C$ is braided equivalent to Deligne tensor product of two Ising categories by \cite[Theorem 3.13]{DrGNO2}.
\end{proof}
\begin{coro}\label{premodular16}Let $\C$ be a strictly weakly integral braided fusion category of FP-dimension $16$, then $\C$ is a generalized Tambara-Yamagami fusion category or $\C\cong\mathcal{TY}(\Z_2\times \Z_2\times \Z_2,\tau,\mu)$ or $\C\cong \A_1\boxtimes\A_2$, where $\A_i$ are  Ising categories.
\end{coro}
\begin{proof}Since $\C$ is nilpotent and weakly integral, direct computation shows  dimension grading group $E=\Z_2$. If $FPdim(\C_{pt})=8$, then $FPdim(X)\in{\{1,\sqrt{2}}\}$ or ${\{1,2\sqrt{2}}\}$ $(\forall X\in\Q(\C))$. In the first case $\C$ is a generalized Tambara-Yamagami fusion category, if $\C$ contains a self-dual non-invertible simple object, then  $\C\cong\I\boxtimes\D$, $\D$ is a  pointed fusion category; if not, $\C\cong\I_3$ or $\C\cong\I_2\boxtimes\B$ by \cite[Theorem 5.5]{DNS}. In the second case, $\C$ is a Tambara-Yamagami fusion category, then it follows from \cite[Theorem 1.2]{Si} that $\C\cong \mathcal{TY}(\Z_2\times \Z_2\times \Z_2,\tau,\mu)$.

If $FPdim(\C_{pt})=4$, then $\C$ contains a unique simple object $V$ of FP-dimension $2$, and $G(\C)=\Z_2\times\Z_2=\langle g\rangle\times\langle h\rangle$ by \cite[Theorem 1.2]{Si}, since $\C_{int}$  is a Tambara-Yamagami fusion category of FP-dimension $8$. Note that $FPdim(X)\in{\{1,\sqrt{2},2}\}$ or ${\{1,2,2\sqrt{2}}\}$ ($X\in\Q(\C)$) by \cite[Corollary 5.3]{GN}.
If $FPdim(X)\in{\{1,2,2\sqrt{2}}\}$, then $\C_{pt}\subseteq\C_{ad}=\C_{int}$, by
\cite[Corollary 3.26]{DrGNO2} $\C_{pt}\cong Rep(G)$ is a Tannakian subcategory, and $\C_G\cong \I$ since $\C$ is strictly weakly integral. Let $Y$ be the unique non-integral simple object of $\C$, $F:\C\to\C_G$ be the forgetful functor, then $F(Y)=2X$ and $[F(Y)\otimes F(Y):I]=4$, where $X$ is the generator of $\I$; while $[F(Y\otimes Y):I]>4$, this is impossible.

Let $FPdim(X)\in{\{1,\sqrt{2},2}\}$, $\forall X\in\Q(\C)$. Since $\C$ contains a unique simple object $V$ of FP-dimension $2$, $sVec\nsubseteq\C'$ by \cite[Lemma 5.4]{Mu}. If $\C'=Vec$, that is, $\C$ is non-degenerate, then $\C$ is equivalent to Deligne tensor product of two Ising categories by  Lemma \ref{nondegen16}. If $FPdim(\C')\geq2$, Let $X$ be an arbitrary simple object of FP-dimension $\sqrt{2}$, assume $X\otimes X^*=I\oplus g$, then $X$ generates a generalized Tambara-Yamagami fusion category. Then \cite[Lemma 3.4]{DNS} shows that $\theta_g=-1$, and $s_{g,V}=-2$. As $sVec\nsubseteq\C'$, $\C'$ must be    Tannakian  category of FP-dimension $2$. Let $\B=\langle g\rangle\cong sVec$, then centralizer $\B'$   is a  strictly weakly integral fusion category of FP-dimension $8$ by Theorem \ref{dimensionform}  as $V\notin\B'$. So $\B'$   is a generalized Tambara-Yamagami fusion category of FP-dimension $8$,   $\B'\cong\I_2$ by Remark \ref{premodular8}. Note that M\"{u}ger center of $\B'$ is exactly symmetric fusion category $\B\vee\C'$ by Remark \ref{Tannmugdeequiv},  this contradicts to Remark \ref{N-Ising}, however.
\end{proof}

\begin{prop}Let $\C$ be  a slightly degenerate  fusion category of FP-dimension $32$. Then $\C\cong sVec\boxtimes\A$, where  $\A$ is a   non-degenerate fusion category.
\end{prop}
\begin{proof}By Lemma \ref{integralslightly2^n}, if $\C$ is integral then $\C$ is pointed, so $\C\cong sVec\boxtimes\A$ by Proposition \ref{slighpointed}. Assume $\C$ is strictly weakly integral, and  $4|FPdim(\C_{pt})|16$. If $FPdim(\C_{pt})=4$,  it follows from Theorem \ref{dimensionform} that
$FPdim(\C_{ad})=\left\{
\begin{array}{ll}
8, & \hbox{$\C_{ad}\cap\C'=Vec$,} \\
16, & \hbox{$\C_{ad}\cap\C'=sVec$.}
\end{array}
\right.$. In the first case $8|FPdim(\C_{pt})$; in the second case $\C_{ad}$ contains exactly three simple objects of FP-dimension $2$, this contradicts to Lemma \ref{slightMugercenter}. If $FPdim(\C_{pt})=16$, then $\C_{int}=\C_{pt}$, as in Proposition \ref{slight16} $\C$ is a generalized Tambara-Yamagami fusion category, then Corollary \ref{generalizedTYstru} implies the classification.

If $FPdim(\C_{pt})=8$, by Theorem \ref{dimensionform} $FPdim(\C_{ad})=4$ or $8$. Then  a direct computation shows that $FPdim(\C_{int})=16$, and there exists non-integral $X,Y\in\Q(\C)$ such that $X\otimes Y\in\Q(\C)$. Otherwise, $\C$ contains a  generalized Tambara-Yamagami fusion category of FP-dimension $24$, impossible. We show $FPdim(\C_{ad})=4$.

If $FPdim(\C_{ad})=8$, then Lemma \ref{slightMugercenter} and Proposition \ref{admugercenter} imply that $\C_{ad}=\C_{pt}$ is symmetric, and $U_\C=\Z_4$ or $\Z_2\times\Z_2$, $FPdim(X)\in{\{1,\sqrt{2},2}\}$ ($X\in\Q(\C)$) by \cite[Corollary 5.3]{GN}. If $U_\C=\Z_4=\langle x\rangle$, then $\C=\oplus^3_{i=0}\C_{x^i}$ is generated by a non-integral simple object $X$ by \cite[Theorem 4.7]{Na}. Assume $X\in\C_{x}$, then $\C_{int}=\C_0\oplus\C_{x^2}$ for simple objects of FP-dimension $2$ are self-dual by Lemma \ref{slightMugercenter}, therefore $X\otimes X\in\C_{x^2}$ must be a simple object. While $X\otimes X^*=I\oplus g$, for some $g\neq I$ with $g\otimes g=I$, this means $[(X\otimes X)\otimes (X\otimes X)^*:I]=2$, impossible.

If $U_\C=\Z_2\times\Z_2$, then non-integral simple objects are not self-dual, otherwise $\C\cong\A_1\boxtimes\A_2\boxtimes sVec$, where $\A_1,\A_2$ are equivalent to Ising categories, but the universal grading of $\A_1\boxtimes\A_2\boxtimes sVec$ is $\Z_2\times\Z_2\times\Z_2$.  Then $\C$ contains a strictly weakly integral braided fusion category $\B$ of FP-dimension   $16$, and $\C_{pt}=\C_{ad}\subseteq\B$, so $\B$  is a generalized Tambara-Yamagami fusion category, since $U_\C=\Z_2\times\Z_2$.   In particular, $\C'\subseteq\B$ as $\B_{pt}=\C_{pt}$, so $\B$ is slightly degenerate by \cite[Lemma 4.12]{DNS}. But by  Proposition \ref{slight16} $\B\cong sVec\boxtimes \I\boxtimes\D_1$, then $\C\cong sVec\boxtimes \I\boxtimes\D_1\boxtimes\D_2$, where $\D_1,\D_2$ are pointed fusion categories of FP-dimension $2$, this is impossible as $U_\C\ncong\Z_2\times\Z_2$.

If $FPdim(\C_{pt})=8, FPdim(\C_{ad})=4$.  On the one hand, it follows from Proposition \ref{slightdim8} and \cite[Theorem 3.13]{DrGNO2} that $\C\cong\A_1\boxtimes\A_2\boxtimes sVec$ if simple objects of FP-dimension $\sqrt{2}$ are self-dual, where $\A_i$ are equivalent to Ising categories. On the other hand, assume that simple  objects of FP-dimension $\sqrt{2}$ are not self-dual. By \cite[Lemma 3.4]{DNS}, $\C_{ad}$ contains a fusion subcategory $\A$ which is braided equivalent  to $sVec$,  obviously $X\notin\A'$ if $FPdim(X)=2$ and $X\in\Q(\C)$. Therefore, Theorem \ref{dimensionform} shows that centralizer $\A'$ is a generalized Tambara-Yamagami fusion category of FP-dimension $16$, and $\C'\subseteq\C_{pt}\subseteq\A'$. Arguments in the previous paragraph says that this is impossible.
\end{proof}
\begin{coro}Let $\C$ be a strictly weakly integral non-degenerate fusion category of FP-dimension $32$, then $\C\cong \A_1\boxtimes\D$ or $\A_2\boxtimes\A_3\boxtimes\B$, where $\A_1,\A_2,\A_3$ are  braided equivalent to Ising categories, $\B,\D$ are pointed non-degenerate  fusion categories.
\end{coro}
\begin{proof}Since $\C$ is strictly weakly integral and modular, $FPdim(\C_{pt})|16$; as $\C_{ad}'=\C_{pt}$ \cite[Corollary 6.8]{GN} and integral non-invertible simple objects have FP-dimensions $2$ or $4$, we see $4|FPdim((\C_{ad})_{pt})$. Then $FPdim(\C_{pt})=16$ or $8$ or $4$. If $FPdim(\C_{pt})=16$, then $\C\cong\I\boxtimes\D$ by \cite[Proposition 5.1]{DongNa},  where $\D$ is pointed.

If $FPdim(\C_{pt})=4$ and $FPdim(\C_{ad})=8$. Then $\C_{pt}\subseteq\C_{ad}$ is symmetric, by \cite[Theorem 1.2]{Si} $\C_{ad}\cong \mathcal{TY}(\Z_2\times \Z_2,\tau,\mu)$ and $U_\C=\Z_2\times \Z_2$ by \cite[Theorem 6.2]{GN}; moreover balancing equation \cite[Proposition 8.13.8]{EGNO} implies that $\C_{pt}\cong Rep(\Z_2\times \Z_2)$ is Tannakian.  By \cite[Corollary 5.3]{GN}, $FPdim(X)\in{\{1,\sqrt{2},2,2\sqrt{2}}\}$, $X\in\Q(\C)$.  If there exists a simple object of FP-dimension $2\sqrt{2}$, as $U_\C=\Z_2\times\Z_2$ then $\C$ contains a pre-modular fusion category $\B$ of FP-dimension $16$, and $FPdim(Y)\in{\{1,2,2\sqrt{2}}\}$ ($Y\in\Q(\B)$); while Corollary \ref{premodular16} shows there is no such a fusion category. Then $FPdim(X)\in{\{1,2,\sqrt{2}}\}$, $X\in\Q(\C)$. Let $\B$ be a fusion subcategory generated by simple object of FP-dimension $\sqrt{2}$, then $\B$ is a generalized Tambara-Yamagami fusion category, so $\B_{ad}\cong sVec$
by \cite[Lemma 3.4]{DNS}, this contradicts to that $\C_{pt}$ is Tannakian.

If $FPdim(\C_{pt})=8$ and $FPdim(\C_{ad})=4$. Since $\C$ is not a generalized Tambara-Yamagami fusion category, there must exists simple object of FP-dimension $2$, and $FPdim(X)\in{\{1,\sqrt{2},2}\}$, $\forall X\in\Q(\C)$ by \cite[Corollary 5.3]{GN}, then $FPdim(\C_{int})=16$. Let $\B$ be a fusion subcategory generated by simple object of FP-dimension $\sqrt{2}$, then $\B$ is a generalized Tambara-Yamagami fusion category, so $\B_{ad}\cong sVec$ by \cite[Lemma 3.4]{DNS}. Therefore, centralizer $\B_{ad}'$ of $\B_{ad}$ in $\C$ is a slightly degenerate fusion category of FP-dimension $16$ by  Theorem \ref{dimensionform} and Remark \ref{Tannmugdeequiv}.   Meanwhile, it follows from Proposition  \ref{slight16} that $\B_{ad}'\cong sVec\boxtimes\D$, where $\D$ is a non-degenerate pointed fusion category.

Then $\D\cong\I\boxtimes\D_1$, where $\D_1$ is a non-degenerate pointed fusion category of FP-dimension $2$. Indeed, if $\D$ is pointed, then $\C\cong\D\boxtimes\I$ is a generalized Tambara-Yamagami fusion category, this is a contradiction. Therefore, $\C\cong\D_1\boxtimes\A_1\boxtimes\A_2$,  where $\A_1,\A_2$ are braided equivalent to Ising categories.
\end{proof}

Last we classify slightly degenerate fusion category of FP-dimension $8d$ and $16d$, where $d>1$ is an odd square-free integer.

\begin{lemm}\label{lemmastrictly8d}Let  $\C$ be a  strictly weakly integral fusion category of FP-dimension $8d$, d is an odd square-free number, then $FPdim(\C_{pt})=4m$, $m\mid d$.
\end{lemm}
\begin{proof}By Lemma \ref{general2^ndsldege}, $8\nmid FPdim(\C_{pt})$, and $FPdim(\C_{pt})=2^tm$, where $d=ms$, $t=1,2$. If $t=1$, then $FPdim(\C_{ad})=\left\{
 \begin{array}{ll}
4s, & \hbox{$\C'\cap\C_{ad}=Vec$,} \\
8s, & \hbox{$\C'\subseteq\C_{ad}$.}
\end{array}
\right.$. Since non-invertible integral simple objects of $\C$ have FP-dimension $2$ by Corollary \ref{sqdimension}, by calculating $FPdim(\C_{ad})$  we see $4|FPdim((\C_{ad})_{pt})$, this is impossible.
\end{proof}

Before we give the classification of slightly degenerate fusion categories of FP-dimension $8d$, let us recall classification of modular fusion categories $\C$ of FP-dimension $4d$, where $d$ is an odd square-free integer. This has been obtained in \cite[Theorem 3.1]{BGNPRW1}, up to braided equivalence, there are three classes: $\C$ is pointed or $\C\cong\I\boxtimes \A$ or $\C\cong \mathcal{TY}(\Z_k,\tau,\mu)^{\Z_2}\boxtimes\B$, where $\I$ is an Ising category, $\A$ and $\B$ are pointed modular fusion categories, $k\mid d$.
By \cite[Proposition 5.1]{GNN}, if fusion category $\mathcal{TY}(\Z_n,\tau,\mu)$ admits a $\Z_2$-crossed braiding structure, then $\mathcal{TY}(\Z_n,\tau,\mu)^{\Z_2}$ is modular if and only if $n$ is odd.

\begin{theo}\label{slightly8d}Slightly degenerate fusion category of FP-dimension $8d$ is equivalent to $sVec\boxtimes\A$ or $\D^{\Z_2}\boxtimes\B$, where $\A$ is a modular fusion category of FP-dimension $4d$, $\B$ is a pointed modular fusion category, $\D$ is a generalized Tambara-Yamagami fusion category of FP-dimension $4k$,  $G(\D)\cong \Z_{2k}$ and $\Gamma\cong\Z_{2k}$ or $\Z_{k}$, $k|d$ and $d$ is an odd square-free integer.
\end{theo}
\begin{proof}If $\C$ is integral, then $\C$ is pointed  by Lemma \ref{integralslightly2^nd}, by Proposition \ref{slighpointed} $\C\cong sVec\boxtimes \A$. If $\C$ is strictly weakly integral, by Lemma \ref{lemmastrictly8d} and Theorem \ref{dimensionform}, $FPdim(\C_{pt})=4m$,
 $FPdim(\C_{ad})=\left\{
 \begin{array}{ll}
2s, & \hbox{$\C'\cap\C_{ad}=Vec$,} \\
4s, & \hbox{$\C'\subseteq\C_{ad}$.}
\end{array}
\right.$, $ms=d$.
If $s=1$, then $\C$ is nilpotent, \cite[Theorem 6.12]{DrGNO2} and Proposition \ref{slightdim8} imply that $\C\cong sVec\boxtimes\A$, where $\A:=\I\boxtimes\D$ is a non-degenerate fusion category.

If  $s>1$, since $d$ is square-free, as in Theorem \ref{slightstrictweakly} we can reduce the classification to when $d=s$. Hence we assume $FPdim(\C_{pt})=4$, $FPdim(\C_{ad})=\left\{
 \begin{array}{ll}
2d, & \hbox{$\C'\cap\C_{ad}=Vec$,} \\
4d, & \hbox{$\C'\subseteq\C_{ad}$.}
\end{array}
\right.$
In both cases, $\C_{pt}\cap\C_{ad}$ contains a Tannakian subcategory $\E$ of FP-dimension $2$. In fact, since $\C$ is strictly weakly integral, $\C_{ad}$ is integral and not pointed, by Corollary \ref{sqdimension}  any non-invertible simple object $X\in\C_{ad}$ has FP-dimension $2$, then there exists a unique $g\in \C_{pt}$ such that $I\neq g\subseteq X\otimes X^*$. By Proposition \ref{admugercenter} and balancing equation \cite[Proposition 8.13.8]{EGNO} we have $2=s_{g,X}=2\theta^{-1}_g$, i.e. $\theta_g=1$, so $(\C_{ad})_{pt}$ contains a  Tannakian subcategory $\E$ generated by $g$.

Consider the M\"{u}ger center of $\E$, we see $FPdim(\E')=4d$ and $(\E')_{\Z_2}$ is slightly degenerate fusion category by Remark \ref{Tannmugdeequiv}. Then $(\E')_{\Z_2}\cong Vec^\omega_{\Z_d}$ by Lemma \ref{general2^ndsldege}, \cite[Theorem 7.2]{NNW} shows that $\E'$ is group-theoretical. Therefore  $\C_{int}=\E'$ and $E\cong\Z_2$, that is, there only exists one non-integral FP-dimension by
\cite[Theorem 3.10]{GN}. In fact, for $FPdim(\C_{pt})=4$, the action of $\Z_2$ on non-trivial subgroups of $\Z_d$ must be non-trivial, then we have  $\C_{ad}\cong Rep(D_{2d})$ if $\C'\cap\C_{ad}=Vec$,  and similarly $\C_{int}\cong Rep(D_{4d})$ as braided fusion categories.

Case $(i)$: $FPdim(X)\in{\{1,2,\sqrt{k},2\sqrt{k}}\}$, $X\in\Q(\C)$, $k>1$ and $k\mid d$. By Proposition \ref{orthomatrix} orthogonality of characters $s_{I,-}$ and $s_{g,-}$ implies $\theta_{g\otimes Y}=-\theta_Y$, $\forall Y\in\Q(\C)$ and $FPdim(Y)\notin\Z$, as $\C_{int}\cong Rep(D_{4d})$. We show there is no simple object with FP-dimension $2\sqrt{k}$.  On the contrary, assume $FPdim(Y)=2\sqrt{k}$, and $FPdim(X)=\sqrt{k}$, $X,Y\in\Q(\C)$. Then $Y\otimes Y^*=I\oplus g\chi\oplus_{i}A_i$, where $A_i$ are simple objects of FP-dimension $2$, and  $V\otimes X\notin\Q(\C)$ with $FPdim(V)=2$. If $V\otimes X\in\Q(\C)$, since $g\otimes V\otimes X=V\otimes X$, $\theta_{g\otimes V\otimes X}=-\theta_{V\otimes X}$, impossible. Therefore, simple objects with FP-dimension $\sqrt{k}$ and integral simple objects generate a fusion subcategory $\B$, while $4d<FPdim(\B)<8d$ and $FPdim(\B)|8d$, contradiction. Then there is no simple object of FP-dimension $\sqrt{k}$, then $FPdim(\C)=4d+c(2\sqrt{k})^2$, where $c$ is the number of non-integral simple object, so $c>0$. However by Lemma \ref{slightMugercenter} $2|c$, which means $2|d$, impossible.

Therefore, $FPdim(X)\in{\{1,2,\sqrt{k}}\}$, $X\in\Q(\C)$. For any two simple objects $X,Y$ of FP-dimension $\sqrt{k}$, $X\otimes Y$ contains a unique invertible simple objects $a$ as $FPdim(X\otimes Y)$ is odd, which means $G(\C)$ acts transitively on the set of non-integral simple objects. Hence there are exactly four non-integral simple objects, by calculating the FP-dimension of $\C$, we see $k=d$. Note $(\C_{int})_{\Z_2}\cong Vec^\omega_{\Z_{2d}}$, hence  all integral simple objects of $\C$ are mapped to direct sums of invertible objects via forgetful functor  $F:\C\to\C_{\Z_2}$, $\Z_2=\langle g\rangle$, then $\C_{\Z_2}$ is a generalized Tambara-Yamagami fusion category.

If there exists a non-integral simple object $Y$ with $Y\cong Y^*$, since ${\{a\otimes Y|a\in G(\C)}\}$ are all non-integral simple objects of $\C$, $\C_{\Z_2}$ has two self-dual simple objects of FP-dimension $\sqrt{d}$. Therefore, $\C_{\Z_2}\supseteq\mathcal{TY}(\Z_d,\tau,\mu)$, which is generated by a non-integral simple object, then $\C$ contains a braided fusion subcategory $\mathcal{TY}(\Z_d,\tau,\mu)^{\Z_2}$. By \cite[Proposition 5.1]{GNN} $\mathcal{TY}(\Z_d,\tau,\mu)^{\Z_2}$ is a modular fusion category, so $\C\cong sVec\boxtimes\mathcal{TY}(\Z_d,\tau,\mu)^{\Z_2}$ by \cite[Theorem 3.13]{DrGNO2}.

If there  exists a non-self-dual simple object $Y$ of FP-dimension $\sqrt{d}$, then $Y\otimes Y=a\oplus_iA_i $ we have $a=g\chi$. Indeed if $a=g$ or $\chi$, then $g\otimes Y=Y^*$ or $\chi\otimes Y=Y^*$, while $\theta_{g\otimes Y}=\theta_{\chi\otimes Y}=-\theta_Y$, this is impossible. Therefore  $Y$, $Y^*$, $g\otimes Y$, $g\otimes Y^*$ are non-isomorphic simple objects of $\C$. Again the  de-equivariantization $\D:=\C_{\Z_2}$ is a generalized Tambara-Yamagami fusion category of FP-dimension $4d$, $\D$ has two non-integral simple objects $Y,Y^*$ of FP-dimension $\sqrt{d}$ and $G(\D)=\Z_{2d}$, so $\C\cong \D^{\Z_2}$.

Case $(ii)$: $FPdim(X)\in{\{1,2,\sqrt{2k}}\}$, $X\in\Q(\C)$, $k\mid d$. Let $FPdim(Y)=\sqrt{2k}$, then $Y\otimes Y^*=I\oplus\chi g\oplus_ia_iX_i$ and $g\otimes Y\cong \chi \otimes Y$, as $\theta_{g\otimes Y}=\theta_{\chi\otimes Y}=-1$, $\C_{int}=\C_{ad}$.  We show $FPdim(Y)\neq\sqrt{2}$ if $d>1$. Indeed, if for any non-integral simple objects  $Y$ and  $Z$ with FP-dimension $\sqrt{2}$, we have $Y\otimes Y^*=Z\otimes Z^*=I\oplus g\chi$, then  $Y\otimes Z\in\C_{pt}$. Then all the non-integral simple objects generate a generalized Tambara-Yamagami fusion subcategory $\D$ with $FPdim(\D)=4d+4$, then $4d|(4d+4)$ and $(4d+4)|8d$, which means $d=1$, i.e. $\C\cong \I\boxtimes sVec$ by Proposition \ref{slightdim8}.

Next, we show non-integral simple objects  are self-dual. Let $Y$ be a self-dual non-integral simple object and $V$ be a simple object of FP-dimension $2$, as  $V\otimes Y=g\otimes V\otimes Y$, then $V\otimes Y=Z\oplus g\otimes Z$. As $V,Y$ are self-dual and $\theta_{g\otimes Z}=-\theta_Z$, so $Z$ is also self-dual. Then self-dual simple objects generate a fusion category of $\C$, which contains $\C_{int}$. While $FPdim(\C_{int})=4d$, therefore, either all the  non-integral simple objects of $\C$ are self-dual or not self-dual. If all the non-integral simple objects of $\C$ are not self-dual, then $FPdim(\C)=4d+(\sqrt{2k})^24l$, where $l$ is the number of quadruples ${\{Y, Y^*, g\otimes Y, g\otimes Y^*}\}$, then $2|d$, impossible. Hence all the non-integral simple objects  are self-dual, and there are $\frac{2d}{k}$ non-integral simple objects.  Since $(\C_{int})_{\Z_2}\cong Vec^\omega_{\Z_{2d}}$,  $\C\cong \D^{\Z_2}$, where $\D$ is a generalized Tambara-Yamagami fusion category with $\frac{d}{k}$ self-dual simple objects of FP-dimension $\sqrt{2k}$, and $G(\D)\cong \Z_{2d}$. It follows from \cite[Proposition 5.1]{GNN} and \cite[Theorem 3.1]{BGNPRW1} that $\C$ does not contains non-integral modular fusion subcategories.
\end{proof}
Modular fusion categories of FP-dimensions $8d$ have been classified in \cite{BPR}, where $d$ is an square-free odd integer. By using same method of Theorem \ref{slightly8d}, we can also have
\begin{theo}\label{slightly16d}Slightly degenerate  fusion categories of FP-dimensions $16d$ are equivalent to $sVec\boxtimes\A_1$ or $\A_2\boxtimes\B$ or $\D^{\Z_2}\boxtimes\A_3$, where $\A_i$ are non-degenerate fusion categories, $\B$ is a slightly degenerate strictly weakly integral fusion category of FP-dimension $8k$, $\D$ is a generalized Tambara-Yamagami fusion category of FP-dimension $8l$, where $d$ is a square-free odd integer, $k\mid d$ and $l\mid d$.
\end{theo}
\begin{proof}If $\C$ is integral, then $\C$ is pointed  by Lemma \ref{integralslightly2^nd}. Let $\C$ be a slightly degenerate strictly weakly integral fusion category of FP-dimensions $16d$, by Corollary \ref{sqdimension}, integral simple objects of $\C$ have FP-dimension $1$ or $2$. As  in Theorem \ref{slightly8d} we assume $FPdim(\C_{pt})|16$, by Lemma \ref{general2^ndsldege}, $16\nmid FPdim(\C_{pt})$, so $FPdim(\C_{pt})=8$. Indeed, if $FPdim(\C_{pt})=4$, then by Theorem \ref{dimensionform} $FPdim(\C_{ad})=4d$ or $8d$; in the first case $8|FPdim(\C_{pt})$, while in the second case the rank of $\C_{ad}$ is $2d+3$, which contradicts to Lemma \ref{slightMugercenter}.

Since $FPdim(\C_{pt})=8$, as in Theorem \ref{slightly8d}, $\C_{ad}\supseteq\E=Rep(\Z_2)$, where $\E$ is a Tannakian subcategory and $\Z_2=\langle g\rangle$. By Remark \ref{Tannmugdeequiv} we see $\E'_{\Z_2}$ is a slightly degenerate  of FP-dimension $4d$, therefore $\E'_{\Z_2}\cong sVec\boxtimes\D$ by Proposition \ref{general2^ndsldege}, where $\D$ is a pointed modular fusion category of FP-dimension $2d$, so $\E'$ is group-theoretical by \cite[Theorem 7.2]{NNW}, then $\E=\C_{int}$ and $E=\Z_2$. Computations in Theorem \ref{slightly8d} shows that $g\otimes V=V$ if $V$ is an integral non-invertible simple object and that $\theta_{g\otimes V}=-\theta_V$ for all non-integral simple objects. However, now $\C_{int}\ncong Rep(D_{8d})$ as $G(Rep(D_{8d}))\cong\Z_2\times\Z_2$.

Case $(i)$: $FPdim(X)\in{\{1,2,\sqrt{k},2\sqrt{k}}\}$, $\forall X\in\Q(\C)$. The same way as Theorem \ref{slightly8d}, one  can show that either all the non-integral simple objects have FP-dimension $\sqrt{k}$ or $2\sqrt{k}$. If $FPdim(X)\in{\{1,2,2\sqrt{k}}\}$, then $FPdim(\C_{ad})=4d$ and $|U_\C|=4$ as $\C'\subseteq\C_{ad}$. If $U_\C=\Z_2\times\Z_2$, then $\C$ contains a strictly weakly integral fusion category $\A$ of FP-dimension $8d$ and $\C_{ad}\subseteq\A$, so $FPdim(\A')=4$; note $\C'\subseteq\A'$, therefore either $\A$ is slightly degenerate or $\A'$ contains a non-trivial Tannakian subcategory. While the first case can not happen by Theorem \ref{slightly8d} and considering the de-equivariantization the second case implies $\A$ is integral, impossible. If $U_\C=\Z_4$, then $\C=\oplus^3_{i=0}\C_{a^i}$ and $\C_{pt}\subseteq\C_{0}\oplus\C_{a^2}=\A$, indeed one can show $\A=\C_{int}$. Then $\C_a$ contains $\frac{d}{k}$ non-integral simple objects, this is impossible as $\chi\in\C_0$ and $\chi\otimes X\in\C_0$ for any $X\in\C_a$. Therefore  $FPdim(X)\in{\{1,2,\sqrt{k}}\}$, $\forall X\in\Q(\C)$, and $k=d$ as Theorem \ref{slightly8d}. If $\C$ contains a self-dual non-integral simple object, then $\C\cong sVec\boxtimes \mathcal{TY}(\Z_d,\tau,\mu)^{\Z_2}\boxtimes\D$ by \cite[Proposition 5.1]{GNN}, $\D$ is pointed modular of FP-dimension $2$; if not, then either $\C$ contains an even metaplectic modular fusion category of FP-dimension $8d$ \cite[Theorem 4.18]{BPR} or $\C\cong\D^{\Z_2}$, where $\D$ is a generalized Tambara-Yamagami fusion category with $G(\D)=\Z_2\times\Z_{2d}$ and $\Gamma=\Z_d$.

Case $(ii)$: $FPdim(X)\in{\{1,2,\sqrt{2k},2\sqrt{2k}}\}$, $\forall X\in\Q(\C)$. Like Theorem \ref{slightly8d} then one  can show that all the non-integral simple objects have FP-dimension $\sqrt{2k}$, and $k>1$ if $d\neq1$; when $d=1$, this is Proposition \ref{slight16}. Therefore $\C\cong sVec\boxtimes\I\boxtimes \A$ or $\C\cong\D^{\Z_2}$, where $\A$ is a pointed modular fusion category of FP-dimension $2d$, $\D$ is a generalized Tambara-Yamagami fusion category with $G(\D)=\Z_2\times\Z_{2d}$ and $\Gamma=\Z_{2d}$. Then it follows from \cite[Theorem 3.1]{BGNPRW1} and \cite[Theorem 4.18]{BPR} that $\C$ does not contain  modular fusion subcategories with $FPdim(X)\in{\{1,2,\sqrt{2k}}\}$ and $k>1$.
\end{proof}
\section*{Acknowledgements}
The author is deeply grateful to Professor V. Ostrik for suggesting this problem and weekly insightful conversations. Without his influence, this article would not have been written. He  thanks
A. Schopieray for comments on previous version of this paper. He also thanks the referee for pointing errors in previous version. This paper was written during a visit of the author at University of Oregon supported by China Scholarship Council (Grant No. 201806140143), he  appreciates the Department of Mathematics for their warm hospitality.

\bigskip
\author{Zhiqiang Yu\\ \thanks{Email:\,zhiqyu-math@hotmail.com}
\\{\small School of Mathematical Sciences,  East China Normal University,
Shanghai 200241, China}
}

\begin{thebibliography}{50}
\bibitem[1]{B}P. Bruillard, Pre-modular categories of rank 4,  New York J. Math. $\mathbf{22}$ (2016), 775-800.
\bibitem[2]{BGHNPRW}P. Bruillard, C. Galindo, T. Hagge, S-H. Ng, J.Plavnik, E.Rowell, Z. Wang, Fermionic modular categories and the 16-fold way, J. Math. Phys, $\mathbf{58}$ (2017), 041704.
\bibitem[3]{BGNPRW1}P. Bruillard, C. Galindo,  S-H. Ng, J. Plavnik, E.Rowell, Z. Wang, On the classification of weakly integral modular categories, J. Pure Appl. Algebra, $\mathbf{220}$ (2016), 2364-2388.
\bibitem[4]{BGNPRW}P. Bruillard, C.Galindo, S-H. Ng, J. Plavnik, E. Rowell, Z. Wang, Classification of super modular categories by rank, Algebr. Reresent. Theory, to appear, arXiv:1705.05293v2.
\bibitem[5]{BPR}P. Bruillard,  J.Plavnik, E.Rowell, Modular categories of dimension $p^3m$ with $m$ square-free, Pro. Amer. Math. Soc. $\mathbf{147}$ (2019), no. 1, 21-34.
\bibitem[6]{DNO}A. Davydov, D. Nikshych and V. Ostrik, On the structure of Witt group of  braided fusion categories, Sel. Math. New. Ser. $\textbf{19}$ (2013), no. 1, 237-269.
\bibitem[7]{De}P. Deligne: Cat\'{e}gories tensorielles. Moscow Math. J. $\mathbf{2}$ (2002), no. 2, 227-248.
\bibitem[8]{DLD} J. Dong, L. Li and L. Dai, Integral almost square-free modular categories, J. Algebra Appl. $\textbf{16}$ (2017), no. 6, 1750104.
\bibitem[9]{DongNa} J. Dong and S. Natale, On the classification of almost square-free integral modular categories, Algebr. Reresent. Theory. $\textbf{21}$ (2018), 1353-1368.
\bibitem[10]{DNS}J. Dong, S. Natale and H. Sun, A class of prime fusion categories of dimension $2^N$, arXiv:1910.07034v2.
\bibitem[11]{DrGNO1}V. Drinfeld, S. Gelaki, D. Nikshych and V. Ostrik, Group-theoretical properties of nilpotent modular categories, arXiv: 0704.0195.
\bibitem[12]{DrGNO2}V. Drinfeld, S. Gelaki, D. Nikshych and V. Ostrik, On braided fusion categories \uppercase\expandafter{\romannumeral 1}, Sel. Math. New. Ser. $\textbf{16}$ (2010), no. 2, 1-119.
\bibitem[13]{EGNO}P. Etingof, S. Gelaki, D. Nikshych and V. Ostrik, Tensor categories, AMS Mathematical Surveys and Monographs $\textbf{205}$ (2015).
\bibitem[14]{ENO1}P. Etingof, D. Nikshych and V. Ostrik, On fusion categories, Ann.of Math. $\textbf{162}$ (2005), no. 2, 581-642.
\bibitem[15]{ENO2}P. Etingof, D. Nikshych and V. Ostrik, Weakly group-theoretical and solvable fusion categories, Adv. Math. $\textbf{226}$  (2011), no. 1, 176-205.
\bibitem[16]{GN} S. Gelaki and D. Nikshych, Nilpotent fusion categories, Adv. Math. $\textbf{217}$ (2008), no. 3, 1053-1071.
\bibitem[17]{GNN}S. Gelaki, D. Nikshych, D. Naidu, Centers of graded fusion categories, Algebra Number Theory, $\textbf{3}$  (2009), 959-990.
\bibitem[18]{Lip}J. Liptrap, Generalized Tambara-Yamagami categories, J. Algebra, to appear, arXiv: 1002.3166v2.
\bibitem[19]{Mu}M. M\"{u}ger, Galois theory for braided tensor categories and modular closure, Adv. Math. $\textbf{150}$ (2000), no. 2, 151-201.
\bibitem[20]{NNW} D. Naidu, D. Nikshych and S. Witherspoon, Fusion subcategories of representation
categories of twisted quantum doubles of finite groups, Int. Math. Res. Not. $\textbf{2009}$ (2009), no. 22, 4183-4219.
\bibitem[21]{Na}S. Natale, Faithful simple object, orders and gradings of fusion categories, Algebr. Geom. Topol. $\textbf{13}$  (2013), 1489-1511.
\bibitem[22]{Na2}S. Natale, The core of a weakly group-theoretical  braided fusion
category, Internat. J. Math. $\textbf{29}$ (2018), no. 2, 1850012, 23 pp.
\bibitem[23]{Si}J. Siehler, Braided near-group categories, arXiv: math/0011037v1.
\bibitem[24]{TY}D. Tambara and S. Yamagami, Tensor categories with fusion rules of self-duality for finite abelian groups, J. Algebra, $\textbf{209}$  (1998), 692-707.
\end{thebibliography}
\end{document}